\documentclass[secthm,seceqn,amsthm,ussrhead,reqno, 12pt]{amsart}
\usepackage[utf8]{inputenc}
\usepackage[english]{babel}
\usepackage{amssymb,amsmath,amsthm,amsfonts,xcolor,enumerate,hyperref,    comment,longtable,cleveref}

\usepackage{times}
\usepackage{cite}
\usepackage{pdflscape}
\usepackage{ulem}
\usepackage[mathcal]{euscript}
\usepackage{tikz}
\usepackage{hyperref}
\usepackage{cancel}
\usepackage{stmaryrd}
\usepackage{float}
\usepackage{caption}
\def\remove#1{}
\usetikzlibrary{arrows}

\newtheorem*{theoremA}{Theorem A}
\newtheorem*{theoremB}{Theorem B}
\newtheorem*{theoremC}{Theorem C}

\newtheorem{theorem}{Theorem}
\newtheorem{lemma}[theorem]{Lemma}
\newtheorem{proposition}[theorem]{Proposition}
\newtheorem{remark}[theorem]{Remark}

\newtheorem{definition}[theorem]{Definition}
\newtheorem{corollary}[theorem]{Corollary}

\usepackage{stmaryrd}
\usepackage{xcolor}

\begin{document}

\noindent{\Large 
The algebraic and geometric classification of \\ 
transposed Poisson algebras}\footnote{
The first part of this work is supported by the Spanish Government through the Ministry of Universities grant `Margarita Salas', funded by the European Union - NextGenerationEU; by  FCT UIDB/00212/2020 and UIDP/00212/2020. 
The second part of this work is supported by RSF 22-11-00081. 
}

 {\bf
 Patr\'{i}cia Damas Beites$^{a}$,
      Amir Fern\'andez Ouaridi$^{b,c}$ \&
 Ivan Kaygorodov$^{a,d,e}$ \\

\smallskip
}

{\tiny
$^a$ Departamento de Matemática and Centro de Matemática e Aplicações, Universidade da Beira Interior, Covilh\~{a}, Portugal

 $^{b}$   Centro de Matemática, Universidade de Coimbra, Coimbra, Portugal

$^{c}$ University of Cadiz, Puerto Real, Spain

$^{d}$   Saint Petersburg  University, Russia

$^{e}$ Moscow Center for Fundamental and Applied Mathematics, Moscow,   Russia 
\smallskip

 E-mail addresses:

     Patr\'{i}cia Damas Beites (pbeites@ubi.pt)

Amir Fern\'andez Ouaridi (amir.fernandez.ouaridi@gmail.com)

 Ivan Kaygorodov (kaygorodov.ivan@gmail.com) 

 \smallskip

}

\noindent {\bf Abstract.}
{\it 
The algebraic and geometric classification of all complex $3$-dimensional transposed Poisson algebras is obtained.
Also, we discuss strong special $3$-dimensional transposed Poisson algebras.
 }

\ 



\section*{Introduction}

Poisson algebras arose from the study of Poisson geometry in the 1970s and have appeared in an extremely wide range of areas in mathematics and physics, such as Poisson manifolds, algebraic geometry, operads, quantization theory, quantum groups, and classical and quantum mechanics. The study of Poisson algebras also led to other algebraic structures, such as 
$F$-manifold algebras,
Novikov-Poisson algebras,
Double Poisson algebras,
Poisson $n$-Lie algebras, etc \cite{vd,blsm17,d19}.
The study of all possible Poisson structures with a certain Lie or associative part is an important problem in the theory of Poisson algebras \cite{YYZ07}.
Recently, a dual notion of the Poisson algebra (transposed Poisson algebra), by exchanging the roles of the two binary operations in the Leibniz rule defining the Poisson algebra, has been introduced in the paper \cite{bai20} of Bai, Bai, Guo, and Wu. 
They have shown that the transposed Poisson algebra defined in this way not only shares common properties with the Poisson algebra, including the closure undertaking tensor products and the Koszul self-duality as an operad but also admits a rich class of identities. More significantly, a transposed Poisson algebra naturally arises from a Novikov-Poisson algebra by taking the commutator Lie algebra of the Novikov algebra. 
Unital transposed Poisson algebras are studied in \cite{bfk22}.
The ${\rm Hom}$- and ${\rm BiHom}$-versions of transposed Poisson algebras are considered in \cite{ hom,bihom}.
Some new examples of transposed Poisson algebras are constructed by applying the Kantor product of multiplications on the same vector space \cite{FK21}. 
More recently, in a paper by Ferreira, Kaygorodov and  Lopatkin,
a relation between $\frac{1}{2}$-derivations of Lie algebras and 
transposed Poisson algebras has been established \cite{FKL}.

The algebraic classification (up to isomorphism) of algebras of dimension $n$ from a certain variety
defined by a certain family of polynomial identities is a classic problem in the theory of non-associative algebras.
There are many results related to the algebraic classification of small-dimensional algebras in many varieties of
associative and non-associative algebras.
So, algebraic classifications of 
$2$-dimensional algebras \cite{petersson},
$3$-dimensional evolution algebras \cite{ccsmv},
$3$-dimensional anticommutative algebras \cite{japan,ikv19},
$4$-dimensional division algebras \cite{Ernst,erik},
$5$-dimensional commutative nilpotent algebras \cite{jkk21}
and 
$8$-dimensional dual Mock Lie algebras  \cite{ckls20} have been given.
Section \ref{alg3dim} is devoted to the complete algebraic classification of non-isomorphic complex $3$-dimensional transposed Poisson algebras.
 To obtain this classification, we will use the algebraic classification of suitable Lie algebras and associative commutative algebras; and the method of describing all transposed Poisson algebra structures on a given Lie algebras (the present method has been developed in \cite{FKL}).

The study of special and non-special algebras starts from the theory of Jordan algebras.
It is known, that the class of special Jordan algebras (i.e. embedded into associative algebras relative to the multiplication $x \circ y =xy+yx$) is a quasivariety, but it is not a variety of algebras \cite{sverchkov}. 
It is known that each Novikov-Poisson algebra under commutator product on non-associative multiplication gives a transposed Poisson algebra \cite{bai20}.
Let us say that a transposed Poisson algebra is special if it can be embedded into a  Novikov-Poisson algebra relative to the commutator bracket. 
Similarly, let us say that a transposed Poisson  algebra is  $D$-special  (from ``differentially'') if it embeds into 
a commutative algebra with a derivation relative to the bracket 
$[x,y]=\mathfrak D(x)y-x\mathfrak D(y).$ Obviously, every $D$-special transposed Poisson algebra is a special one.
On the other hand,   the class of all special Gelfand-Dorfman
algebras (i.e. embedded into Poisson algebras with derivation relative to the multiplication $x \circ y =xd(y)$) is closed with respect to homomorphisms and thus forms
a variety \cite{ks21}. Also known as each two-generated Jordan algebra is special \cite{shirshov};
each one-generated Jordan dialgebra is special \cite{vv12};
each $2$-dimensional  Gelfand-Dorfman algebra is special \cite{ks21}.
   Section \ref{spe3dim} is devoted to the description of all complex strong  $D$-special 
   (i.e. isomorphic to a commutative algebra with a derivation and the bracket given by 
$[x,y]=\mathfrak D(x)y-x\mathfrak D(y)$) $2$- and $3$-dimensional transposed Poisson algebras.

 Geometric properties of a variety of algebras defined by a family of polynomial identities have been an object of study since the 1970s (see, \cite{wolf2,   chouhy,   fkkv22,     gabriel,      cibils,  shaf, GRH}). 
 Gabriel described the irreducible components of the variety of $4$-dimensional unital associative algebras~\cite{gabriel}.  
 Cibils considered rigid associative algebras with $2$-step nilpotent radical \cite{cibils}.
 Grunewald and O'Halloran computed the degenerations for the variety of $5$-dimensional nilpotent Lie algebras~\cite{GRH}. 
 All irreducible components of  $2$-step nilpotent commutative and anticommutative algebras have been described in \cite{shaf,ikp20}.
Chouhy proved that in the case of finite-dimensional associative algebras,
 the $N$-Koszul property is preserved under the degeneration relation~\cite{chouhy}.
  The study of degenerations of algebras is very rich and closely related to deformation theory, in the sense of Gerstenhaber \cite{ger63}.
The geometric classification is given for many varieties of non-associative algebras (see, for example, \cite{fkkv22,ckls20,fkk21,ak21,afk21,gkp21,fkkv22,ikv19} and references therein).
Degenerations have also been used to study a level of complexity of an algebra~\cite{wolf2,gorb93}.
Section \ref{geo} is devoted to the complete geometric classification of complex $3$-dimensional transposed Poisson algebras.

 \newpage

  \section{The algebraic classification of $3$-dimensional transposed Poisson algebras}
 \label{alg3dim}
   \subsection{Preliminaries}\label{prem}
Although all algebras and vector spaces are considered over the complex field.
The definition of transposed Poisson algebra was given in a paper by Bai, Bai, Guo, and Wu \cite{bai20}.
The concept of $\frac{1}{2}$-derivations as a particular case of $\delta$-derivations was presented in an paper of Filippov \cite{fil1}
(see also \cite{z10} and references therein).

\begin{definition}\label{tpa}
Let ${\mathfrak L}$ be a vector space equipped with two nonzero bilinear operations $-\cdot-$ and $[\cdot,\cdot].$
The triple $({\mathfrak L},\cdot,[\cdot,\cdot])$ is called a transposed Poisson algebra if $({\mathfrak L},\cdot)$ is a commutative associative algebra and
$({\mathfrak L},[\cdot,\cdot])$ is a Lie algebra that satisfies the following compatibility condition
\begin{center}
$2z\cdot [x,y]=[z\cdot x,y]+[x,z\cdot y].$ 
\end{center}
\end{definition}

The last relation  is called the transposed Leibniz rule because the
roles played by the two binary operations in the Leibniz rule in a
Poisson algebra are switched. Further, the resulting operation is
rescaled by introducing a factor 2 on the left-hand side.

\begin{definition}\label{12der}
Let $({\mathfrak L}, [\cdot,\cdot])$ be an algebra with multiplication $[\cdot,\cdot],$ $\varphi$ be a linear map
and $\phi$ be a bilinear map.
Then $\varphi$ is a $\frac{1}{2}$-derivation if it satisfies
\begin{center}
$\varphi[x,y]= \frac{1}{2} \big([\varphi(x),y]+ [x, \varphi(y)] \big);$
\end{center}
 $\phi$ is a $\frac{1}{2}$-biderivation if it satisfies
\begin{longtable}{crl}
$\phi([x,y],z)$&$=$&$ \frac{1}{2} \big( 
 [\phi(x,z),y] +[x,\phi(y,z)]\big),$\\ 
$\phi(x,[y,z])$&$=$&$ \frac{1}{2} \big([\phi(x,y),z]+ 
[y,\phi(x,z)] \big).$
\end{longtable}

\end{definition}

Summarizing Definitions \ref{tpa} and \ref{12der}, we have the following key Remark.
\begin{remark}\label{glavlem}
Let $({\mathfrak L},\cdot,[\cdot,\cdot])$ be a transposed Poisson algebra 
and $z$ an arbitrary element from ${\mathfrak L}.$
Then the right multiplication $R_z$ in the associative commutative algebra $({\mathfrak L},\cdot)$ gives a $\frac{1}{2}$-derivation of the Lie algebra $({\mathfrak L}, [\cdot,\cdot])$
and $-\cdot-$ gives a $\frac{1}{2}$-biderivation of $({\mathfrak L},[\cdot,\cdot])$
 satisfying the
identities $x \cdot y = y \cdot x$ and $x \cdot (y\cdot z) = (x \cdot y)\cdot z$ for any $x, y, z.$ 
Reciprocally, for any $\frac{1}{2}$-biderivation $D: {\mathfrak L} \times {\mathfrak L} \to {\mathfrak L}$ 
satisfying $D(x, y) = D(y, x)$ and 
$D(x, D(y, z)) = D(D(x, y), z)$  one gets a Poisson transposed algebra in an obvious way.
\end{remark}

The main example of $\frac{1}{2}$-derivations is the multiplication by an element from the ground field.
Let us call such $\frac{1}{2}$-derivations as trivial
$\frac{1}{2}$-derivations.
As a consequence of the following Remark,
we are not interested in trivial $\frac{1}{2}$-derivations.

\begin{remark}\label{princth}
Let ${\mathfrak L}$ be a Lie algebra without non-trivial $\frac{1}{2}$-derivations.
Then every transposed Poisson algebra structure defined on ${\mathfrak L}$ is trivial.
\end{remark}

\subsection{Isomorphism problem for transposed Poisson algebras on a certain Lie algebra}

\begin{definition} 
Let $({\mathfrak L}_1, \cdot_1, [\cdot,\cdot]_1)$ and $({\mathfrak L}_2, \cdot_2, [\cdot,\cdot]_2)$ be two transposed Poisson algebras.
Then $({\mathfrak L}_1, \cdot_1, [\cdot,\cdot]_1)$ and $({\mathfrak L}_2, \cdot_2, [\cdot,\cdot]_2)$ are isomorphic if and only if there exists
a liner map $\varphi$  such that 
\begin{center}
    $\varphi([x,y]_1)=[\varphi(x),\varphi(y)]_2, \ \varphi(x\cdot_1 y)=\varphi(x) \cdot_2 \varphi(y).$
\end{center}
\end{definition} 
 
Our main strategy for classifying all non-isomorphic transposed Poisson algebra structures on a certain Lie algebra 
 $({\mathfrak L}, [\cdot,\cdot])$ is as follows.
\begin{enumerate}
\item Find all automorphisms ${\rm Aut}({\mathfrak L}, [\cdot,\cdot])$. 
\item Consider the multiplication table of $({\mathfrak L}, \cdot )$ under the action of elements from  ${\rm Aut}({\mathfrak L}, [\cdot,\cdot])$ and separate all non-isomorphic cases.
\end{enumerate}

\subsection{$\frac{1}{2}$-derivations of $3$-dimensional Lie algebras 
and transposed Poisson   algebras} 
To describe all $3$-dimensional transposed Poisson algebras we are using the standard way:
obtain the classification of all $\frac{1}{2}$-derivations of $3$-dimensional Lie algebras 
and construct all possible transposed Poisson structures on these algebras by using Remark \ref{glavlem}.

\subsubsection{Classification of $3$-dimensional Lie algebras}
In the subsequent result, the classification of $3$-dimensional complex Lie algebras is recalled (the presented classification was used in \cite{ikv19}).

\begin{theorem}\label{theoclassificLie}
Let $({\mathfrak L}, [\cdot,\cdot])$ be a $3$-dimensional complex nonzero Lie algebra,
then one and only one of the following possibilities holds up to isomorphism:
\begin{longtable*}{cclll}

$\mathfrak{h}$ & $:$&
    $[ e_1 , e_2 ] = e_3,$\\

$\mathfrak{g}_1$ & $:$&
    $[e_1, e_{3}]=e_1,$ &  $[e_2, e_{3}]=e_2,$\\

$\mathfrak{g}_2^\alpha$ &$:$& 
    $[e_1, e_{3}]=e_1+e_2,$ & $[e_2, e_3]= \alpha e_2,$\\
    
$\mathfrak{sl}_2$ & $:$&
    $[e_1, e_{2}]=e_3,$ &  $[e_{1}, e_3]=-e_2,$ & $[e_2, e_{3}]=e_1.$

\end{longtable*}

Between these algebras, the only non-trivial isomorphisms are $\mathfrak{g}_2^\alpha\cong \mathfrak{g}_2^\beta$ if and only if $\alpha=\beta^{-1}$.
 
\end{theorem}

Ferreira, Kaygorodov and Lopatkin proved that there 
are no non-trivial transposed Poisson algebra structures defined on a complex semisimple finite-dimensional Lie algebra, result which applies to  $\mathfrak{sl}_2$. Denote by ${\rm T}_{01}$ the trivial transposed Poisson algebra defined on $\mathfrak{sl}_2$. The transposed Poisson algebra structures defined on $\mathfrak{h}$ 
were studied, using $\frac{1}{2}$-biderivations, by  Yuan and Hua in \cite[Theorem 4.5]{yh21}, obtaining the algebras:
\begin{enumerate}
    \item ${\rm T}_{02}:\left\{ 
\begin{tabular}{l}
$e_2 \cdot  e_2 = e_3,$\\ 
$[ e_1 , e_2 ] = e_3.$ 
\end{tabular}%
\right. $

    \item ${\rm T}_{03}^{\alpha}:\left\{ 
\begin{tabular}{l}
$e_1 \cdot  e_2 = \alpha e_3,$ \\ 
$[ e_1 , e_2 ] = e_3.$
\end{tabular}%
\right. $

    \item ${\rm T}_{04}^{\alpha}:\left\{ 
\begin{tabular}{l}
$e_1 \cdot e_2 = \alpha e_3,  e_2 \cdot e_2  = e_1,$ \\ 
$[ e_1 , e_2 ] = e_3.$
\end{tabular}%
\right. $

    \item ${\rm T}_{05}:\left\{ 
\begin{tabular}{l}
$e_1 \cdot e_1 = e_3,  e_1 \cdot e_2 = e_1,  e_2 \cdot e_2  = e_2,  e_2 \cdot e_3 = e_3,$ \\ 
$[ e_1 , e_2 ] = e_3.$
\end{tabular}%
\right. $

    \item ${\rm T}_{06}:\left\{ 
\begin{tabular}{l}
$e_1 \cdot e_2= e_1, e_2 \cdot e_2 = e_2 , e_2 \cdot e_3= e_3,$ \\ 
$[ e_1 , e_2 ] = e_3.$
\end{tabular}%
\right. $
\end{enumerate}

Hence, only the classification of the transposed Poisson algebras defined on $\mathfrak{g}_1$ and $\mathfrak{g}_2^\alpha$ are missing.
    
\subsubsection{Description of transposed Poisson algebra structure defined on   $\mathfrak{g}_1$}

\begin{remark}\label{propderiv3dim2Lie}
  Let $\varphi$ be a $\frac{1}{2}$-derivation of $\mathfrak{g}_1$. Then 
$$\varphi(e_1)=\beta_{33} e_1, \ \varphi(e_2)= \beta_{33} e_2, \ \varphi(e_3)= \beta_{31} e_1 + \beta_{32} e_2 + \beta_{33} e_3.      
$$
\end{remark}

\begin{remark}\label{autg1}
  Let $\phi$ be an automorphism of $\mathfrak{g}_1$. Then 
$$
\phi(e_1)=\lambda_{11} e_1+\lambda_{21} e_2, \ 
\phi(e_2)=\lambda_{12} e_1+\lambda_{22} e_2, \ 
 \ \phi(e_3)= \lambda_{13} e_1+\lambda_{23} e_2+e_3,      
$$
where $\lambda_{11}\lambda_{22}\neq \lambda_{21}\lambda_{12}.$
\end{remark}

\begin{proposition} \label{TPAmult4}
Let $(\mathfrak{L}, \cdot, [\cdot,\cdot])$ be a transposed Poisson algebra structure defined on   $\mathfrak{g}_1$.
Then $(\mathfrak{L}, \cdot, [\cdot,\cdot])$ is not a Poisson algebra and it is isomorphic to only one of the following algebras:
\begin{enumerate}[(1)]
\item ${\rm T}_{07}^{\alpha}:\left\{ 
\begin{tabular}{l}
$e_1 \cdot e_3 = \alpha e_1$, $e_2 \cdot e_3 = \alpha e_2$, $e_3 \cdot e_3 = \alpha e_3,$ \\ 
$[e_1, e_{3}]=e_1,$  $[e_2, e_{3}]=e_2.$%
\end{tabular}%
\right. $
 
\item ${\rm T}_{08}:\left\{ 
\begin{tabular}{l}
$e_3 \cdot e_3 = e_1,$ \\ 
$[e_1, e_{3}]=e_1,$  $[e_2, e_{3}]=e_2.$%
\end{tabular}%
\right. $
\end{enumerate}
where the parameter $\alpha\in \mathbb{C}$.
\end{proposition}

\begin{proof}
We aim to describe the multiplication $-\cdot-$. By Remark \ref{glavlem}, for every $e_k$ there is an associated $\frac{1}{2}$-derivation $\varphi_k$ of $(\mathfrak{L}, [\cdot,\cdot])$ such that $\varphi_j(e_i)=e_i \cdot e_j = \varphi_i(e_j)$. From Remark \ref{propderiv3dim2Lie}, 
\begin{center}
    $\varphi_i(e_1) = \beta_{33}^i e_1$, $\varphi_i(e_2) = \beta_{33}^i e_2$, $\varphi_i(e_3) = \beta_{31}^i e_1 +\beta_{32}^i e_2 + \beta_{33}^i e_3$
\end{center} and
\begin{longtable*}{rcccccl}
$\beta_{33}^1 e_2$ & $= $& $\varphi_1 (e_2)$ & $=$ & $\varphi_2(e_1)$ &$ =$& $\beta_{33}^2 e_1$,  \\
$\beta_{31}^1 e_1 + \beta_{32}^1 e_2 +\beta_{33}^1 e_3$ & $=$ & $\varphi_1(e_3)$ & $=$ & $\varphi_3(e_1)$ & $=$ & $\beta_{33}^3 e_1,$ \\
$\beta_{33}^3 e_2$ & $= $& $\varphi_2(e_3)$ & $=$ & $\varphi_3(e_2)$ & $=$ & $\beta_{31}^2 e_1 + \beta_{32}^2 e_2 +\beta_{33}^2 e_3.$ 
\end{longtable*}
Hence, the commutative multiplication $-\cdot-$ is defined by
\begin{longtable*}{lcl}
$e_1 \cdot e_3$ & $=$ & $\beta_{33}^3 e_1,$ \\
$e_2 \cdot e_3$ & $=$ & $\beta_{33}^3 e_2,$ \\
$e_3 \cdot e_3$ & $=$ & $\beta_{31}^3 e_1 + \beta_{32}^3 e_2 +\beta_{33}^3 e_3.$
\end{longtable*}
Through straightforward calculations, it is possible to conclude that $-\cdot-$ is associative too.
Let us now separate all non-isomorphic cases.
Under the action of an automorphism of the Lie algebra $(\mathfrak{L},  [\cdot,\cdot]),$
given in Remark \ref{autg1}, we rewrite the multiplication table of $(\mathfrak{L}, \cdot)$ by the following way:

\begin{longtable*}{lcl}
$e_1 \cdot e_3$ & $=$ & $\beta_{33}^3 e_1,$ \\
$e_2 \cdot e_3$ & $=$ & $\beta_{33}^3 e_2,$ \\
$e_3 \cdot e_3$ & $=$ & 
$\frac{\beta^3_{31}\lambda_{22}-\beta^3_{32}\lambda_{12}+
\beta^3_{33}\lambda_{13}\lambda_{22}-\beta^3_{33}\lambda_{12}\lambda_{23}}{\lambda_{11}\lambda_{22}- \lambda_{21}\lambda_{12}} e_1 - 
\frac{\beta^3_{31}\lambda_{21}-\beta^3_{32}\lambda_{11}+
\beta^3_{33}\lambda_{13}\lambda_{21}-\beta^3_{33}\lambda_{11}\lambda_{23}}{\lambda_{11}\lambda_{22}- \lambda_{21}\lambda_{12}} e_2 +\beta_{33}^3 e_3.$
\end{longtable*}

Suppose $(\beta_{31}^3,\beta_{32}^3,\beta_{33}^3)\neq 0$, otherwise we have the zero algebra.
Let us consider two cases.
\begin{enumerate}

    \item  If
    $\beta_{33}^3\neq0,$ then by choosing 
   \begin{center} $\lambda_{13}=-\frac{\beta_{31}^3}{\beta_{33}^3},$
    $\lambda_{23}=-\frac{\beta_{32}^3}{\beta_{33}^3},$
    $\lambda_{11}=\lambda_{22}=1$ and $\lambda_{12}=\lambda_{21}=0,$ \end{center}we have

\begin{longtable*}{lcllcllcl}
$e_1 \cdot e_3$ & $=$ & $\beta_{33}^3 e_1,$ &
$e_2 \cdot e_3$ & $=$ & $\beta_{33}^3 e_2,$ &
$e_3 \cdot e_3$ & $=$ & $\beta_{33}^3 e_3.$
\end{longtable*} 
           \item  If
    $\beta_{33}^3=0$ and $(\beta_{31}^3,\beta_{32}^3)\neq 0,$ then by choosing some suitable  $\lambda_{ij}$ we have 
    \begin{longtable*}{lcllcllcl}
$e_3 \cdot e_3$ & $=$ & $ e_1.$
\end{longtable*}
 
\end{enumerate}

It is easy to see that the obtained algebras are  not isomorphic. From these cases, we obtain the algebras ${\rm T}_{07}^{\alpha}$ and ${\rm T}_{08}$, respectively.
\end{proof}

\subsubsection{Description of transposed Poisson algebra structure defined on   $\mathfrak{g}_2^\alpha$}

\begin{remark}\label{propderiv3dim2Lieb}
Let $\varphi$ be a $\frac{1}{2}$-derivation of $\mathfrak{g}_2^\alpha$. 
\begin{enumerate}
\item If $\alpha   \neq 0, \frac{1}{2}, 2$, then
\begin{center} 
$\varphi(e_1)=\beta_{33}e_1, \ \varphi(e_2)= \beta_{33}e_2, \ \varphi(e_3)=\beta_{31}e_1+\beta_{32}e_2+\beta_{33}e_3.$ 
\end{center}
\item  If $\alpha=\frac{1}{2}$, then
\begin{center}
$\varphi(e_1)=(\beta_{33}-2\beta_{21}) e_1 -4\beta_{21}e_2, \ \varphi(e_2)= \beta_{21}e_1 + (2\beta_{21}+\beta_{33})e_2,$ \end{center}
\begin{center}
$\varphi(e_3)=\beta_{31}e_1+\beta_{32}e_2+\beta_{33}e_3.$
\end{center}
\item If $\alpha=2$, then
\begin{center}
$\varphi(e_1)=\beta_{33}e_1 + \beta_{12} e_2, \ \varphi(e_2)= \beta_{33}e_2,   \ \varphi(e_3)=\beta_{31}e_1+\beta_{32}e_2+\beta_{33}e_3.$
\end{center}
\item If $\alpha=0$, then
\begin{center} 
$\varphi(e_1)=\beta_{33}e_1 -(\beta_{22}-\beta_{33})e_2, \ \varphi(e_2)= \beta_{22}e_2,   \ \varphi(e_3)=\beta_{31}e_1+\beta_{32}e_2+\beta_{33}e_3.$
\end{center}

\end{enumerate}
\end{remark}

\begin{proof}
Let $\varphi(e_i)=\beta_{i1} e_1 + \beta_{i2} e_2 + \beta_{i3} e_3$. From here and
\begin{longtable*}{rcl}
$0$& $=$&$ [\varphi(e_1), e_2] +[e_1, \varphi(e_2)]  = \beta_{23} e_1 -  (\alpha\beta_{13} - \beta_{23}) e_2,$\\

$2\varphi(e_1+e_2)$&$=$&$   [\varphi(e_1), e_3] +[e_1, \varphi(e_3)]  = 
(\beta_{11}+\beta_{33}) e_1 + (\beta_{11} + \alpha \beta_{12}+ \beta_{33}) e_2,$\\

$2\alpha \varphi( e_2)$&$=$&$  [\varphi(e_2), e_3] +[e_2, \varphi(e_3)]  = 
 \beta_{21} e_1 +  (\beta_{21} + \alpha \beta_{22}+ \alpha \beta_{33}) e_2,$
\end{longtable*}
we obtain the following system of equations
\begin{equation*}
\left\{
\begin{array}{lll}
\beta_{23}= 0 \\
\beta_{13}= 0 \\ 
\beta_{11} + \beta_{33}= 2(\beta_{11} + \beta_{21}) \\
\beta_{11} + \alpha \beta_{12} + \beta_{33}=2( \beta_{12} + \beta_{22}) \\
2\alpha \beta_{21} = \beta_{21} \\
2\alpha \beta_{22} = \beta_{21} + \alpha\beta_{22}+\alpha\beta_{33}\\
\end{array}
\right. .
\end{equation*} 
That gives the proof of the statement.
\end{proof}

\begin{remark}\label{autg2}
  Let $\phi$ be an automorphism of $\mathfrak{g}_2^\alpha$. 
   
\begin{enumerate}
    \item If $\alpha\neq  -1,$ then 
\begin{center}$
\phi(e_1)=\lambda_{11} e_1+\lambda_{21} e_2, \ 
\phi(e_2)= (\lambda_{11} + \lambda_{21}(\alpha-1)) e_2, \ 
 \ \phi(e_3)= \lambda_{13} e_1+\lambda_{23} e_2+e_3,      
$\end{center}
where $\lambda_{11}(\lambda_{11} + \lambda_{21}(\alpha-1))\neq0.$

  \item If $\alpha= -1,$ then 
\begin{center}$
\phi(e_1)=\lambda_{11} e_1+\lambda_{21} e_2, \ 
\phi(e_2)= (\lambda_{11} -2 \lambda_{21}) e_2, \ 
 \ \phi(e_3)= \lambda_{13} e_1+\lambda_{23} e_2+e_3,            
$\end{center}
where $\lambda_{11}(\lambda_{11} -2 \lambda_{21})\neq0;$
or 
\begin{center}$
\phi(e_1)=\lambda_{11} e_1 + \lambda_{21}e_2, \ 
\phi(e_2)=-2\lambda_{11} e_1-\lambda_{11}e_2, \ 
 \ \phi(e_3)= \lambda_{13} e_1+\lambda_{23} e_2-e_3,      
$\end{center}
where $\lambda_{11}^2\neq 2\lambda_{11}\lambda_{21}.$

\end{enumerate}

\end{remark}

We aim to describe the multiplication $-\cdot-$. By Remark \ref{glavlem}, for every $e_k$ there is an associated $\frac{1}{2}$-derivation $\varphi_k$ of $(\mathfrak{L}, [\cdot,\cdot])$ such that $\varphi_j(e_i)=e_i \cdot e_j = \varphi_i(e_j)$. 
Now we have to consider $3$ cases, by Remark \ref{propderiv3dim2Lieb} and using that $\mathfrak{g}_2^{\frac{1}{2}} \cong \mathfrak{g}_2^2$.

\begin{proposition}
Let $(\mathfrak{L}, \cdot, [\cdot,\cdot])$ be a transposed Poisson algebra structure defined on   $\mathfrak{g}_2^{\alpha\neq 0,\frac{1}{2},2}$.
Then $(\mathfrak{L}, \cdot, [\cdot,\cdot])$ is isomorphic to only one of the following algebras:
\begin{enumerate}

 \item ${\rm T}_{09}^{\alpha\neq0, \frac{1}{2},2, \beta}:\left\{ 
\begin{tabular}{l}
$ e_1 \cdot e_3= \beta  e_1, \ e_2 \cdot e_3 = \beta  e_2,\  e_3 \cdot e_3 = \beta e_3,$ \\ 
 $[e_1, e_{3}]=e_1+e_2,$ $[e_2, e_3]= \alpha e_2.$
\end{tabular}%
\right. $

 \item ${\rm T}_{10}^{\alpha\neq0,\frac{1}{2},2}:\left\{ 
\begin{tabular}{l}
$ e_3 \cdot e_3 = e_2,$ \\ 
 $[e_1, e_{3}]=e_1+e_2,$ $[e_2, e_3]= \alpha e_2.$
\end{tabular}%
\right. $

 \item ${\rm T}_{11}^{\alpha\neq0,\frac{1}{2},2}:\left\{ 
\begin{tabular}{l}
$ e_3 \cdot e_3 = e_1,$ \\ 
 $[e_1, e_{3}]=e_1+e_2,$ $[e_2, e_3]= \alpha e_2.$
\end{tabular}%
\right. $

 \end{enumerate}
 where the parameter $\beta\in \mathbb{C}$. Between these algebras there are precisely the following non-trivial isomorphisms:
 \begin{itemize}
     \item ${\rm T}_{09}^{\alpha_1, \beta_1}\cong {\rm T}_{09}^{\alpha_2, \beta_2}$ if and only if $(\alpha_2, \beta_2)=(\alpha_1, \beta_1)$ or  $(\alpha_2, \beta_2)=(\frac{1}{\alpha_1}, \frac{\beta_1}{\alpha_1})$.
     
     \item ${\rm T}_{11}^{\beta_1}\cong {\rm T}_{11}^{\beta_2}$ if and only if $\beta_2=\beta_1$ or $\beta_2=\frac{1}{\beta_1}$.

 \end{itemize}
\end{proposition}

\begin{proof}
 From Remark \ref{propderiv3dim2Lieb}, 
\begin{center}$\varphi_i(e_1) = \beta_{33}^i e_1$, $\varphi_i(e_2) = \beta_{33}^i e_2$, $\varphi_i(e_3) = \beta_{31}^i e_1 +\beta_{32}^i e_2 + \beta_{33}^i e_3$
\end{center} 
and
\begin{longtable*}{rcccccl}
$\beta_{33}^1 e_2$ & $=$ & $\varphi_1 (e_2)$ & $=$ & $\varphi_2(e_1)$ &$ =$ & $\beta_{33}^2 e_1$,  \\
$\beta_{31}^1 e_1 + \beta_{32}^1 e_2 +\beta_{33}^1 e_3$ & $=$ & $\varphi_1(e_3)$ & $=$ & $\varphi_3(e_1)$ & $=$ & $\beta_{33}^3 e_1,$ \\
$\beta_{33}^3 e_2$ &$= $& $\varphi_2(e_3)$ & $=$ & $\varphi_3(e_2)$ &$ = $& $\beta_{31}^2 e_1 + \beta_{32}^2 e_2 +\beta_{33}^2 e_3.$ \\
\end{longtable*}
Hence, the commutative multiplication $-\cdot-$ is defined by
\begin{eqnarray*}
e_1 \cdot e_3 & = & \beta_{33}^3 e_1, \\
e_2 \cdot e_3 & = & \beta_{33}^3 e_2, \\
e_3 \cdot e_3 & = & \beta_{31}^3 e_1 + \beta_{32}^3 e_2 +\beta_{33}^3 e_3.
\end{eqnarray*}

Through straightforward calculations, it is possible to conclude that $-\cdot-$ is associative too. We are interested in $(\beta_{31}^3,\beta_{32}^3,\beta_{33}^3)\neq 0.$
Let us now separate all non-isomorphic cases.
\medskip
\begin{enumerate}
    \item[\mbox{First,}] 
we will consider only the case $\alpha\neq -1,0,1,\frac{1}{2},2.$
Under the action of an automorphism of the Lie algebra $(\mathfrak{L},  [\cdot,\cdot]),$
given in Remark \ref{autg2}, we rewrite the multiplication table of $(\mathfrak{L}, \cdot)$ by the following way:

\begin{longtable*}{lcl}
$e_1 \cdot e_3$ & $=$ & $\beta_{33}^3 e_1,$ \\
$e_2 \cdot e_3$ & $=$ & $\beta_{33}^3 e_2,$ \\
$e_3 \cdot e_3$ & $=$ & 
$\frac{\beta^3_{31}+\beta^3_{33}\lambda_{13}}{\lambda_{11}} e_1+$\\ &&$
\frac{\beta^3_{31}(\lambda_{11}-\lambda_{22})-\beta^3_{32}\lambda_{11}+\alpha \beta^3_{32}\lambda_{11}+\beta^3_{33}(\lambda_{11}\lambda_{13}-\lambda_{13}\lambda_{22}-\lambda_{11}\lambda_{23})+\alpha \beta^3_{32}\lambda_{11}\lambda_{23}}{(\alpha-1)\lambda_{11}\lambda_{22}} e_2 +\beta_{33}^3 e_3.$
\end{longtable*}

Let us consider the following cases.
\begin{enumerate}
    \item 
    If $\beta_{33}^3\neq0,$ then by choosing 
   \begin{center} $\lambda_{13}=-\frac{\beta_{31}^3}{\beta_{33}^3},$
    $\lambda_{23}=-\frac{\beta_{32}^3}{\beta_{33}^3},$
    $\lambda_{11}=\lambda_{22}=1,$ \end{center}we have

\begin{longtable*}{lcllcllcl}
$e_1 \cdot e_3$ & $=$ & $\beta_{33}^3 e_1,$ &
$e_2 \cdot e_3$ & $=$ & $\beta_{33}^3 e_2,$ &
$e_3 \cdot e_3$ & $=$ & $\beta_{33}^3 e_3.$
\end{longtable*}

       \item 
    If $\beta_{33}^3=0$ and $\beta_{31}^3 = (1-\alpha)\beta_{32}^3,$ 
    then by choosing  $\lambda_{11}= \beta_{32}^3,$ we have 
    \begin{longtable*}{lcllcllcl}
$e_3 \cdot e_3$ & $=$ & $ (1-\alpha) e_1+e_2.$
\end{longtable*}
    
       \item 
    $\beta_{33}^3=0$ and $\beta_{31}^3 \neq (1-\alpha)\beta_{32}^3,$ 
    then by choosing some suitable  $\lambda_{ij},$ we have two following opportunities
\begin{center}
$e_3 \cdot e_3=  e_2;$ or $e_3 \cdot e_3=  e_1.$ 

\end{center}

\end{enumerate}

Denote the corresponding families of transposed Poisson algebras by ${\rm T}_{09}^{\alpha,\beta}$, ${\rm T}_{10*}^{\alpha}$, ${\rm T}_{10}^{\alpha}$ and ${\rm T}_{11}^{\alpha}$, respectively.

\medskip

\item[\mbox{Second,}] 
 we will consider only the case $\alpha=1.$
Under the action of an automorphism of the Lie algebra $(\mathfrak{L},  [\cdot,\cdot]),$
given in Remark \ref{autg2}, we rewrite the multiplication table of $(\mathfrak{L}, \cdot)$ by the following way:

\begin{longtable*}{lcl}
$e_1 \cdot e_3$ & $=$ & $\beta_{33}^3 e_1,$ \\
$e_2 \cdot e_3$ & $=$ & $\beta_{33}^3 e_2,$ \\
$e_3 \cdot e_3$ & $=$ & 
$\frac{\beta^3_{31}+\beta^3_{33}\lambda_{13}}{\lambda_{11}} e_1+ 
\frac{\beta^3_{32}\lambda_{11}-\beta^3_{31}\lambda_{21}-\beta^3_{33}\lambda_{13}\lambda_{21} +\beta^3_{33}\lambda_{11}\lambda_{23}
}{\lambda_{11}^2} e_2 +\beta_{33}^3 e_3.$
\end{longtable*}

Let us consider the following cases.
\begin{enumerate}
    \item 
    $\beta_{33}^3\neq0,$ then by choosing 
   \begin{center} $\lambda_{13}=-\frac{\beta_{31}^3}{\beta_{33}^3},$
    $\lambda_{23}=-\frac{\beta_{32}^3}{\beta_{33}^3},$
    $\lambda_{11}=1, \lambda_{21}=0,$ \end{center}we have

\begin{longtable*}{lcllcllcl}
$e_1 \cdot e_3$ & $=$ & $\beta_{33}^3 e_1,$ &
$e_2 \cdot e_3$ & $=$ & $\beta_{33}^3 e_2,$ &
$e_3 \cdot e_3$ & $=$ & $\beta_{33}^3 e_3.$
\end{longtable*}       
    
       \item 
    $\beta_{33}^3=0,$ 
    then by choosing some suitable  $\lambda_{ij},$ we have two following opportunities
\begin{center}
$e_3 \cdot e_3=  e_2;$ or 
$e_3 \cdot e_3=  e_1.$
\end{center}

\end{enumerate}

From here, we obtain the transposed Poisson algebras ${\rm T}_{09}^{1,\beta}$, ${\rm T}_{10}^{1}$ and ${\rm T}_{11}^{1}$, respectively.

\medskip

\item[\mbox{Third,}] 
we will consider only the case $\alpha=-1.$
Under the action of an automorphism of the first type of the Lie algebra $(\mathfrak{L},  [\cdot,\cdot]),$
given in Remark  \ref{autg2}, we rewrite the multiplication table of $(\mathfrak{L}, \cdot)$ by the following way:

\begin{longtable*}{lcl}
$e_1 \cdot e_3$ & $=$ & $\beta_{33}^3 e_1,$ \\
$e_2 \cdot e_3$ & $=$ & $\beta_{33}^3 e_2,$ \\
$e_3 \cdot e_3$ & $=$ & 
$\frac{\beta^3_{31}+\beta^3_{33}\lambda_{13}}{\lambda_{11}} e_1-$\\ &&$
\frac{\beta^3_{31}\lambda_{11}-2\beta^3_{32}\lambda_{11}+\beta^3_{33}\lambda_{11}\lambda_{13} -\beta^3_{31}\lambda_{22}-\beta^3_{33}\lambda_{13}\lambda_{22}-2\beta^3_{33}\lambda_{11}\lambda_{23}
}{2\lambda_{11} \lambda_{22}} e_2 +\beta_{33}^3 e_3.$
\end{longtable*}

Let us consider the following cases.
\begin{enumerate}
    \item 
    If $\beta_{33}^3\neq0,$ then by choosing 
   \begin{center} $\lambda_{13}=-\frac{\beta_{31}^3}{\beta_{33}^3},$
    $\lambda_{23}=-\frac{\beta_{32}^3}{\beta_{33}^3},$
    $\lambda_{11}=1, \lambda_{21}=0,$ \end{center}we have

\begin{longtable*}{lcllcllcl}
$e_1 \cdot e_3$ & $=$ & $\beta_{33}^3 e_1,$ &
$e_2 \cdot e_3$ & $=$ & $\beta_{33}^3 e_2,$ &
$e_3 \cdot e_3$ & $=$ & $\beta_{33}^3 e_3.$
\end{longtable*}

       \item 
    If $\beta_{33}^3=0$ and $ \beta_{31}^3= 2 \beta_{32}^3,$
    then by choosing      $\lambda_{11}=\beta_{32}^3,$ we have  
    \begin{longtable*}{lcllcllcl}
$e_3 \cdot e_3$ & $=$ & $ 2e_1+e_2.$
\end{longtable*}

       \item 
    If $\beta_{33}^3=0$ and $ \beta_{31}^3\neq 2 \beta_{32}^3,$
    then by choosing some suitable  $\lambda_{ij},$ we have two following opportunities
\begin{center}
$e_3 \cdot e_3=  e_2;$ or 
$e_3 \cdot e_3=  e_1.$
\end{center}

\end{enumerate}

From here, we obtain the transposed Poisson algebras ${\rm T}_{09}^{-1,\beta}$, ${\rm T}_{10}^{-1}$ and ${\rm T}_{11}^{-1}$, respectively.

\end{enumerate}

It is easy to see that 
under the action of an automorphism of the second type, the following algebras

\begin{center}
$e_3 \cdot e_3=  2e_1+e_2$ and  
$e_3 \cdot e_3=  e_2$
\end{center}

are isomorphic.

\medskip

Lastly, since $\mathfrak{g}_2^\alpha\cong \mathfrak{g}_2^\beta$ if and only if $\alpha=\beta^{-1}$, we may find isomorphisms between the transposed Poisson algebras that follow from this study. These isomorphisms are the following:
 \begin{itemize}
     \item ${\rm T}_{09}^{\alpha_1, \beta_1}\cong {\rm T}_{09}^{\alpha_2, \beta_2}$ if and only if $(\alpha_2, \beta_2)=(\alpha_1, \beta_1)$ or  $(\alpha_2, \beta_2)=(\frac{1}{\alpha_1}, \frac{\beta_1}{\alpha_1})$.  For the non-trivial isomorphisms, choose the change of basis:
     $$E_1= \frac{1}{\alpha_1 - 1}e_1, \ E_2= e_1 + \frac{\alpha_1}{\alpha_1-1}e_2, \ E_3= \alpha_1 e_3.$$
     
     \item ${\rm T}_{10*}^{\beta_1}\cong {\rm T}_{10}^{\beta_2}$ if and only if $\beta_2=\frac{1}{\beta_1}$, using the change of basis:
     $$E_1= \frac{1}{\beta_1}e_1, \ E_2= \frac{1-\beta_1}{\beta_1^2}e_1 + \frac{1}{\beta_1^2}e_2, \ E_3= \frac{1}{\beta_1}e_3.$$

     \item ${\rm T}_{11}^{\beta_1}\cong {\rm T}_{11}^{\beta_2}$ if and only if $\beta_2=\beta_1$ or $\beta_2=\frac{1}{\beta_1}$.  For the non-trivial isomorphisms, choose the change of basis:
     $$E_1= \beta_1^2 e_1, \ E_2= (\beta_1 - 1)\beta_1^2 e_1 + \beta_1^3 e_2, \ E_3= \beta_1 e_3 .$$
 \end{itemize}
\end{proof}

\begin{proposition}
Let $(\mathfrak{L}, \cdot, [\cdot,\cdot])$ be a transposed Poisson algebra structure defined on   $\mathfrak{g}_2^2$.
Then $(\mathfrak{L}, \cdot, [\cdot,\cdot])$ is isomorphic to only one of the following algebras:
\begin{enumerate}

 \item ${\rm T}_{09}^{2, \beta}:\left\{ 
\begin{tabular}{l}
$e_1 \cdot e_3= \beta  e_1, \ e_2 \cdot e_3 = \beta  e_2,\  e_3 \cdot e_3 = \beta e_3,$ \\ 
 $[e_1, e_{3}]=e_1+e_2,$ $[e_2, e_3]= 2 e_2.$
\end{tabular}%
\right. $

 \item ${\rm T}_{10}^{2}:\left\{ 
\begin{tabular}{l}
$e_3 \cdot e_3 = e_2 ,$ \\ 
 $[e_1, e_{3}]=e_1+e_2,$ $[e_2, e_3]= 2 e_2.$
\end{tabular}%
\right. $

 \item ${\rm T}_{11}^{2}:\left\{ 
\begin{tabular}{l}
$e_3 \cdot e_3 = e_1 ,$ \\ 
 $[e_1, e_{3}]=e_1+e_2,$ $[e_2, e_3]= 2 e_2.$
\end{tabular}%
\right. $

 \item ${\rm T}_{12}^{\beta}:\left\{ 
\begin{tabular}{l}
$ e_1 \cdot e_1 =  e_2,\ e_1 \cdot e_3 =  \beta e_1, \  e_2 \cdot e_3 =   \beta e_2,  \  e_3 \cdot e_3 =   \beta  e_3,$ \\ 
 $[e_1, e_{3}]=e_1+e_2,$ $[e_2, e_3]= 2 e_2.$
\end{tabular}%
\right. $

 \item ${\rm T}_{13}:\left\{ 
\begin{tabular}{l}
$e_1 \cdot e_1 =  e_2, \  e_3 \cdot e_3 =  e_2 ,$ \\ 
 $[e_1, e_{3}]=e_1+e_2,$ $[e_2, e_3]= 2 e_2.$
\end{tabular}%
\right. $

 \item ${\rm T}_{14}:\left\{ 
\begin{tabular}{l}
$e_1 \cdot e_3 =  e_2, \ e_3 \cdot e_3 =  e_1,$ \\ 
 $[e_1, e_{3}]=e_1+e_2,$ $[e_2, e_3]= 2 e_2.$
\end{tabular}%
\right. $

 \item ${\rm T}_{15}:\left\{ 
\begin{tabular}{l}
$e_1 \cdot e_3 =  e_2 ,$ \\ 
 $[e_1, e_{3}]=e_1+e_2,$ $[e_2, e_3]= 2 e_2.$
\end{tabular}%
\right. $
\end{enumerate}
where the parameter $\beta\in \mathbb{C}$. 
\end{proposition}

\begin{proof}

 From Remark \ref{propderiv3dim2Lieb}, 
\begin{center}
$\varphi_i(e_1) = \beta_{33}^i e_1 +  \beta_{12}^i e_2$, 
$\varphi_i(e_2) = \beta_{33}^i e_2$, 
$\varphi_i(e_3) = \beta_{31}^i e_1 +\beta_{32}^i e_2 + \beta_{33}^i e_3$
\end{center}
and
\begin{longtable*}{rcccccl}
$\beta_{33}^1 e_2$ & $=$ & $\varphi_1 (e_2)$ & $=$ & $\varphi_2(e_1)$ & $=$ & $\beta_{33}^2 e_1 + \beta_{12}^2 e_2$,  \\
$\beta_{31}^1 e_1 + \beta_{32}^1 e_2 +\beta_{33}^1 e_3$ & $=$ & $\varphi_1(e_3)$ & = & $\varphi_3(e_1)$ & $=$ & $\beta_{33}^3 e_1 + \beta_{12}^3 e_2,$ \\
$\beta_{31}^2 e_1 + \beta_{32}^2 e_2 +\beta_{33}^2 e_3$ & $=$ & $\varphi_2(e_3)$ & = & $\varphi_3(e_2)$ & $=$ & $\beta_{33}^3 e_2.$ \\
\end{longtable*}
Hence, the commutative multiplication $-\cdot-$ is defined by
\begin{eqnarray*}
e_1 \cdot e_1 & = &  \beta_{12}^1 e_2, \\
e_1 \cdot e_3 & = & \beta_{33}^3 e_1 +  \beta_{32}^1 e_2, \\
e_2 \cdot e_3 & = & \beta_{33}^3 e_2, \\
e_3 \cdot e_3 & = & \beta_{31}^3 e_1 + \beta_{32}^3 e_2 +\beta_{33}^3 e_3.
\end{eqnarray*}
This multiplication is associative if and only if $\beta_{31}^3\beta_{12}^1=\beta_{32}^1\beta_{33}^3$. 
 Under the action of an automorphism of the Lie algebra $(\mathfrak{L},  [\cdot,\cdot]),$
given in Remark \ref{autg2}, we rewrite the multiplication table of $(\mathfrak{L}, \cdot)$ by the following way:

\begin{longtable*}{lcl}
$e_1 \cdot e_1$ & $=$ & $\frac{\beta_{12}^1 \lambda_{11}^2}{\lambda_{22}} e_2,$ \\
$e_1 \cdot e_3$ & $=$ & $\beta_{33}^3 e_1+\frac{\lambda_{11}(\beta^1_{32}+\beta^1_{12}\lambda_{13})}{\lambda_{22}}e_2,$ \\
$e_2 \cdot e_3$ & $=$ & 
$\beta^3_{33}e_2,$\\

$e_3 \cdot e_3$ & $=$ & 
$\frac{\beta^3_{31}+\beta^3_{33}\lambda_{13}}{\lambda_{11}}e_1+
\frac{\lambda_{11}(\beta^3_{31}+\beta^3_{32}+2\beta^1_{32}\lambda_{13}+\beta^3_{33}\lambda_{13}+\beta^1_{12}\lambda^2_{13}+\beta^3_{33}\lambda_{23})-
(\beta^3_{31}+\beta^3_{33}\lambda_{13})\lambda_{22}      }{\lambda_{11}\lambda_{22}}e_2+\beta_{33}^3e_3.$

\end{longtable*}

Let us consider the following cases.
\begin{enumerate}
    \item $\beta^3_{33}\neq0,$ $\beta_{32}^1=\frac{\beta^3_{31}\beta^1_{12}}{\beta^3_{33}}$ 
    and $\beta^1_{12}\neq0,$ then by choosing
    \begin{center}
         $\lambda_{11}=1,$
         $\lambda_{22}=\beta^1_{12},$
         $\lambda_{13}=-\frac{\beta^3_{31}}{\beta^3_{32}},$
         $\lambda_{23}=\frac{\beta^1_{12}(\beta^3_{31})^2-\beta^3_{32}(\beta^3_{33})^2}{(\beta_{33}^3)^3},$     
    \end{center}
    we have
    \begin{longtable*}{lcllcllcllcl}
$e_1 \cdot e_1$ & $=$ & $ e_2,$ &
$e_1 \cdot e_3$ & $=$ & $ \beta^3_{33} e_1,$&
$e_2 \cdot e_3$ & $=$ & $  \beta^3_{33} e_2,$ &
$e_3 \cdot e_3$ & $=$ & $  \beta^3_{33} e_3.$ 
\end{longtable*}    

   \item $\beta^3_{33}\neq0,$ $\beta_{32}^1=\frac{\beta^3_{31}\beta^1_{12}}{\beta^3_{33}}$ 
    and $\beta^1_{12}=0,$ then by choosing
    \begin{center}
         $\lambda_{11}=1,$
         $\lambda_{22}=1,$
         $\lambda_{13}=-\frac{\beta^3_{31}}{\beta^3_{33}},$
         $\lambda_{23}=-\frac{\beta^3_{32}}{\beta^3_{32}},$     
    \end{center}
    we have
    \begin{longtable*}{lcllcllcllcl}

$e_1 \cdot e_3$ & $=$ & $ \beta^3_{33} e_1,$&
$e_2 \cdot e_3$ & $=$ & $  \beta^3_{33} e_2,$ &
$e_3 \cdot e_3$ & $=$ & $  \beta^3_{33} e_3.$ 
\end{longtable*}

  \item $\beta^3_{33}=0,$ $\beta_{31}^3=0,$ $\beta^1_{12}\neq0$
  and $\beta^1_{12}\beta^3_{32}\neq (\beta^1_{32})^2,$
     then by choosing
    \begin{center}
         $\lambda_{11}=\frac{\sqrt{\beta^1_{12}\beta^3_{32}-(\beta^1_{32})^2}}{\beta^1_{12}},$  
         $\lambda_{22}=-\frac{(\beta^1_{32})^2}{\beta^1_{12}}+\beta^3_{32},$
         $\lambda_{13}=-\frac{\beta^1_{32}}{2\beta^1_{12}},$
         $\lambda_{23}=0,$     
    \end{center}
    we have
    \begin{longtable*}{lcllcllcllcl}
$e_1 \cdot e_1$ & $=$ & $ e_2,$ &
$e_3 \cdot e_3$ & $=$ & $ e_2.$  

\end{longtable*}

  \item $\beta^3_{33}=0,$ $\beta_{31}^3=0,$ $\beta^1_{12}\neq0$
  and $\beta^1_{12}\beta^3_{32}= (\beta^1_{32})^2,$
     then by choosing
    \begin{center}
         $\lambda_{11}=1,$  
         $\lambda_{22}=\beta^1_{12},$
         $\lambda_{13}=-\frac{\beta^1_{32}}{2\beta^1_{12}},$
         $\lambda_{23}=0,$     
    \end{center}
    we have
    \begin{longtable*}{lcllcllcllcl}
$e_1 \cdot e_1$ & $=$ & $ e_2.$ 

\end{longtable*} 

    \item $\beta^3_{33}=0,$ $\beta_{12}^1=0,$ $\beta^1_{32}\neq0$ and $\beta^3_{31}\neq0,$  then by choosing
    \begin{center}
         $\lambda_{11}=\beta^3_{31},$  
         $\lambda_{22}=\beta^1_{32}\beta^3_{31},$
         $\lambda_{13}=\frac{\beta^1_{32}\beta^3_{31}-\beta^3_{31}-\beta^3_{32} }{2\beta^1_{32}},$
         $\lambda_{23}=0,$     
    \end{center}
    we have
    \begin{longtable*}{lcllcllcllcl}
$e_1 \cdot e_3$ & $=$ & $ e_2,$ & 
$e_3 \cdot e_3$ & $=$ & $ e_1.$ & 

\end{longtable*}  
    
  \item $\beta^3_{33}=0,$ $\beta_{12}^1=0,$ $\beta^1_{32}\neq0$ and $\beta^3_{31}=0,$  then by choosing
    \begin{center}
         $\lambda_{11}=1,$  
         $\lambda_{22}=\beta^1_{32},$
         $\lambda_{13}=\frac{\beta^3_{32} }{2\beta^1_{32}},$
         $\lambda_{23}=0,$     
    \end{center}
    we have
    \begin{longtable*}{lcllcllcllcl}
$e_1 \cdot e_3$ & $=$ & $ e_2.$ 
\end{longtable*}

    \item $\beta^3_{33}=0,$ $\beta_{12}^1=0,$ $\beta^1_{32}=0$ and $\beta^3_{31}= -\beta^3_{32},$
   then by choosing  then by choosing
    \begin{center}
         $\lambda_{11}=\beta^3_{32},$  
         $\lambda_{22}=1,$
         $\lambda_{13}=\frac{\beta^3_{32} }{2\beta^1_{32}},$
         $\lambda_{23}=0,$     
    \end{center}
    we have
  \begin{center}
$e_3 \cdot e_3=  -e_1+e_2.$
\end{center}

    \item $\beta^3_{33}=0,$ $\beta_{12}^1=0,$ $\beta^1_{32}=0$ and $\beta^3_{31}\neq -\beta^3_{32},$
   then by choosing some suitable  $\lambda_{ij},$ we have two following opportunities
  \begin{center}
$e_3 \cdot e_3=  e_1;$ or 
$e_3 \cdot e_3=  e_2.$
\end{center}
\end{enumerate}

These cases produce the algebras ${\rm T}_{12}^{\beta\neq0}$, ${\rm T}_{09}^{2,\beta}$, ${\rm T}_{13}$, ${\rm T}_{12}^{0}$, ${\rm T}_{14}$, ${\rm T}_{15}$, ${\rm T}_{10*}^{2}$, ${\rm T}_{10}^{2}$, respectively. Again, recall that ${\rm T}_{10*}^{2}\cong {\rm T}_{10}^{\frac{1}{2}}$.

\end{proof}

\begin{proposition}
Let $(\mathfrak{L}, \cdot, [\cdot,\cdot])$ be a transposed Poisson algebra structure defined on   $\mathfrak{g}_2^0$.
Then $(\mathfrak{L}, \cdot, [\cdot,\cdot])$ is isomorphic to only one of the following algebras:
\begin{enumerate}

 \item ${\rm T}_{09}^{0, \beta}:\left\{ 
\begin{tabular}{l}
$e_1 \cdot e_3= \beta  e_1, \ e_2 \cdot e_3 = \beta  e_2,\  e_3 \cdot e_3 = \beta e_3,$ \\ 
 $[e_1, e_{3}]=e_1+e_2$.
\end{tabular}%
\right. $

 \item ${\rm T}_{10}^0:\left\{ 
\begin{tabular}{l}
$e_3 \cdot e_3 = e_2,$ \\ 
 $[e_1, e_{3}]=e_1+e_2$.
\end{tabular}%
\right. $

 \item ${\rm T}_{11}^0:\left\{ 
\begin{tabular}{l}
$e_3 \cdot e_3 = e_1,$ \\ 
 $[e_1, e_{3}]=e_1+e_2$.
\end{tabular}%
\right. $

 \item ${\rm T}_{16}:\left\{ 
\begin{tabular}{l}
$e_3 \cdot e_3=   e_1+e_2,$ \\ 
 $[e_1, e_{3}]=e_1+e_2$.
\end{tabular}%
\right. $

 \item ${\rm T}_{17}^{\beta}:\left\{ 
\begin{tabular}{l}
$e_1 \cdot e_1=e_2,$ \
$e_1 \cdot e_2=- e_2,$ \
$e_1 \cdot e_3=\beta  e_1,$ \
$e_2 \cdot e_2=   e_2,$ \
$e_2 \cdot e_3= \beta  e_2,$ \
$e_3 \cdot e_3= \beta  e_3,$ \\ 
 $[e_1, e_{3}]=e_1+e_2$.
\end{tabular}%
\right. $

 \item ${\rm T}_{18}:\left\{ 
\begin{tabular}{l}
$e_1 \cdot e_1=e_2,$ \
$e_1 \cdot e_2=- e_2,$ \
$e_2 \cdot e_2=    e_2,$ \
$e_3 \cdot e_3= e_1+e_2,$ \\ 
 $[e_1, e_{3}]=e_1+e_2$.
\end{tabular}%
\right. $

 \item ${\rm T}_{19}^{\gamma}:\left\{ 
\begin{tabular}{l}
$e_1 \cdot e_3= \gamma e_1+ \gamma e_2,$ \
$e_3 \cdot e_3= \gamma  e_3,$ \\ 
 $[e_1, e_{3}]=e_1+e_2$.
\end{tabular}%
\right. $
 \end{enumerate}
 where the parameters $\beta\in\mathbb{C}$ and $\gamma\in\mathbb{C}^*$.
\end{proposition}

\begin{proof}
From Remark \ref{propderiv3dim2Lieb}, 
\begin{center}
    $\varphi_i(e_1) = \beta_{33}^i e_1 + (\beta_{33}^i-\beta_{22}^i)e_2$, $\varphi_i(e_2) = \beta_{22}^i e_2$, $\varphi_i(e_3) = \beta_{31}^i e_1 +\beta_{32}^i e_2 + \beta_{33}^i e_3$
\end{center} and
\begin{longtable*}{rcccccl}
$\beta_{22}^1 e_2$ & $=$ & $\varphi_1 (e_2)$ & $=$ & $\varphi_2(e_1)$ & $= $& $\beta_{33}^2 e_1 + (\beta_{33}^2-\beta_{22}^2)e_2$,  \\
$\beta_{31}^1 e_1 + \beta_{32}^1 e_2 +\beta_{33}^1 e_3$ & $= $& $\varphi_1(e_3)$ & $=$ & $\varphi_3(e_1)$ & $=$ & $\beta_{33}^3 e_1 + (\beta_{33}^3-\beta_{22}^3)e_2,$ \\
$\beta_{31}^2 e_1+\beta_{32}^2 e_2 + \beta_{33}^2 e_3$ & $=$ & $\varphi_2(e_3)$ & $=$ & $\varphi_3(e_2)$ &$ =$ & $\beta_{22}^3 e_2.$ \\
\end{longtable*}
Hence, the commutative multiplication $-\cdot-$ is defined by
\begin{eqnarray*}
e_1 \cdot e_1 & = & \beta_{22}^2 e_2, \\
e_1 \cdot e_2 & = & -\beta_{22}^2 e_2,  \\
e_1 \cdot e_3 & = & \beta_{33}^3 e_1 - (\beta_{22}^3 - \beta_{33}^3)e_2, \\
e_2 \cdot e_2 & = & \beta_{22}^2 e_2, \\
e_2 \cdot e_3 & = & \beta_{22}^3 e_2, \\
e_3 \cdot e_3 & = & \beta_{31}^3 e_1 + \beta_{32}^3 e_2 +\beta_{33}^3 e_3.
\end{eqnarray*} This multiplication is associative if and only if $(\beta_{22}^3)^2-\beta_{33}^3\beta_{22}^3+(\beta_{31}^3-\beta_{32}^3)\beta_{22}^2=0$. 

Under the action of an automorphism of the Lie algebra $(\mathfrak{L},  [\cdot,\cdot]),$
given in Remark \ref{autg2}, we rewrite the multiplication table of $(\mathfrak{L}, \cdot)$ by the following way:

\begin{longtable*}{lcl}
$e_1 \cdot e_1$ & $=$ & $\beta_{22}^2 \lambda_{22}  e_2,$ \\
$e_1 \cdot e_2$ & $=$ & $-\beta_{22}^2 \lambda_{22} e_2,$ \\
$e_1 \cdot e_3$ & $=$ & 
$\beta_{33}^3 e_1+(\beta^3_{33}-\beta^3_{22}+\beta^2_{22}\lambda_{13}-\beta^2_{22}\lambda_{23})e_2,$ \\
$e_2 \cdot e_2$ & $=$ & $\beta^2_{22}\lambda_{22}e_2,$\\

$e_2 \cdot e_3$ & $=$ & $(\beta^3_{22}-\beta^2_{22}\lambda_{13}+\beta^2_{22}\lambda_{23})e_2,$\\

$e_3 \cdot e_3$ & $=$ & 
$\frac{\beta^3_{31}+\beta^3_{33}\lambda_{13}}{\lambda_{11}}e_1+$\\
&&$\frac{(\beta^3_{31}+\beta^3_{33}\lambda_{13})\lambda_{22}+\lambda_{11}(\beta^3_{32}-\beta^3_{31}+\beta^2_{22}\lambda_{13}^2+(2\beta^3_{22}-\beta^3_{33})\lambda_{23}+\beta^2_{22}\lambda_{23}^2+\lambda_{13}(\beta^3_{33}-2\beta^3_{22}-2\beta^2_{22}\lambda_{23}))}
{\lambda_{11}\lambda_{22} }e_2+\beta_{33}^3e_3.$

\end{longtable*}

Let us consider the following cases.
\begin{enumerate}
    \item 
    $\beta_{22}^2\neq0,$ $\beta_{31}^3=\frac{\beta_{33}^3\beta_{22}^3+\beta_{32}^3\beta_{22}^2-(\beta_{22}^3)^2}{\beta_{22}^2}$ and $\beta^3_{33}\neq0,$
    then by choosing 
   \begin{center} 
    $\lambda_{11}=1,$
    $\lambda_{22}=\frac{1}{\beta^2_{22}},$
    $\lambda_{13}=\frac{(\beta_{32}^3)^2 -\beta^2_{22}\beta^3_{32}-\beta^3_{22}\beta^3_{33}}{\beta^2_{22}\beta_{33}^3},$
    $\lambda_{23}=\frac{(\beta_{32}^3)^2 -\beta^2_{22}\beta^3_{32}-2\beta^3_{22}\beta^3_{33}+(\beta^3_{33})^2}{\beta^2_{22}\beta_{33}^3},$
     \end{center}we have

\begin{longtable*}{lcllcllcl}
$e_1 \cdot e_1$ & $=$ & $e_2,$ &
$e_1 \cdot e_2$ & $=$ & $- e_2,$ &
$e_1 \cdot e_3$ & $=$ & $\beta_{33}^3 e_1,$ \\
$e_2 \cdot e_2$ & $=$ & $    e_2,$ &
$e_2 \cdot e_3$ & $=$ & $\beta_{33}^3 e_2,$ &
$e_3 \cdot e_3$ & $=$ & $\beta_{33}^3 e_3.$
\end{longtable*}

    \item 
    $\beta_{22}^2\neq0,$ $\beta_{31}^3=\frac{\beta_{33}^3\beta_{22}^3+\beta_{32}^3\beta_{22}^2-(\beta_{22}^3)^2}{\beta_{22}^2},$ $\beta^3_{33}=0$ and $\beta^2_{22}\beta^3_{32}\neq (\beta^3_{22})^2,$
    then by choosing 
   \begin{center} 
    $\lambda_{11}=-\frac{(\beta^3_{22})^2}{\beta^2_{22}}+\beta^3_{32},$
    $\lambda_{22}=\frac{1}{\beta^2_{22}},$
    $\lambda_{13}=0,$
    $\lambda_{23}=\frac{\beta_{22}^3}{\beta^2_{22}},$
     \end{center}we have 

\begin{longtable*}{lcllcllcllcl}
$e_1 \cdot e_1$ & $=$ & $e_2,$ &
$e_1 \cdot e_2$ & $=$ & $- e_2,$ &
$e_2 \cdot e_2$ & $=$ & $    e_2,$ &
$e_3 \cdot e_3$ & $=$ & $e_1+e_2.$
\end{longtable*}

  \item 
    $\beta_{22}^2\neq0,$ $\beta_{31}^3=\frac{\beta_{33}^3\beta_{22}^3+\beta_{32}^3\beta_{22}^2-(\beta_{22}^3)^2}{\beta_{22}^2},$ $\beta^3_{33}=0$ and $\beta^2_{22}\beta^3_{32}= (\beta^3_{22})^2,$
    then by choosing 
   \begin{center} 
    $\lambda_{11}=1,$
    $\lambda_{22}=\frac{1}{\beta^2_{22}},$
    $\lambda_{13}=0,$
    $\lambda_{23}=\frac{\beta_{22}^3}{\beta^2_{22}},$
     \end{center}we have 

\begin{longtable*}{lcllcllcllcl}
$e_1 \cdot e_1$ & $=$ & $e_2,$ &
$e_1 \cdot e_2$ & $=$ & $- e_2,$ &
$e_2 \cdot e_2$ & $=$ & $    e_2.$ 
\end{longtable*}

 \item 
    $\beta_{22}^2=0,$ $\beta^3_{22}=0$ and $\beta^3_{33}\neq0,$  
    then by choosing 
   \begin{center} 
    $\lambda_{11}=1,$
    $\lambda_{22}=1,$
    $\lambda_{13}=-\frac{\beta^3_{31}}{\beta^3_{32}},$
    $\lambda_{23}=\frac{\beta^3_{32}-2\beta_{31}^3}{\beta^3_{33}},$
     \end{center}we have 

\begin{longtable*}{lcllcllcllcl}
$e_1 \cdot e_3$ & $=$ & $\beta^3_{33} e_1+ \beta^3_{33} e_2,$ &
$e_3 \cdot e_3$ & $=$ & $  \beta^3_{33}  e_3.$ 
\end{longtable*}

 \item 
    $\beta_{22}^2=0,$ $\beta^3_{22}=0,$  $\beta^3_{33}0$ and $\beta^3_{31}=\beta^3_{32},$   
    then by choosing 
   \begin{center} 
    $\lambda_{11}=\beta^3_{32},$
    $\lambda_{22}=1,$
    $\lambda_{13}=0,$
    $\lambda_{23}=0,$
     \end{center}we have 

\begin{longtable*}{lcllcllcllcl}
$e_3 \cdot e_3$ & $=$ & $  e_1 + e_2.$ 
\end{longtable*}

    \item    $\beta_{22}^2=0,$ $\beta^3_{22}=0,$  $\beta^3_{33}0$ and $\beta^3_{31}\neq\beta^3_{32},$   
  then by choosing some suitable  $\lambda_{ij},$ we have two following opportunities
  \begin{center}
$e_3 \cdot e_3=  e_1;$ or 
$e_3 \cdot e_3=  e_2.$
\end{center}

 \item 
    $\beta_{22}^2=0$ and $\beta^3_{22}=\beta^3_{33}\neq 0,$ 
    then by choosing 
   \begin{center} 
    $\lambda_{11}=1,$
    $\lambda_{22}=1,$
    $\lambda_{13}=-\frac{\beta^3_{31}}{\beta^3_{32}},$
    $\lambda_{23}=-\frac{\beta^3_{32}}{\beta^3_{33}},$
     \end{center}we have 

\begin{longtable*}{lcllcllcllcl}
$e_1 \cdot e_3$ & $=$ & $\beta_{33}^3 e_1,$ &
$e_2 \cdot e_3$ & $=$ & $\beta_{33}^3 e_2,$ &
$e_3 \cdot e_3$ & $=$ & $\beta_{33}^3 e_3.$
\end{longtable*}
\end{enumerate}

After constructing the corresponding transposed Poisson algebras from these cases, we obtain the algebras ${\rm T}_{17}^{\beta\neq0}$, ${\rm T}_{18}$, ${\rm T}_{17}^{0}$, ${\rm T}_{19}^{\gamma\neq0}$, ${\rm T}_{16}$, ${\rm T}_{11}^{0}$, ${\rm T}_{10}^{0}$, ${\rm T}_{09}^{0,\beta}$, respectively. 

\end{proof}

\subsection{Classification theorem}
 The results of the previous subsection together with the classification of the commutative associative algebras of dimension three (we recall the classification that was used in \cite{gkp21}), give us the following classification theorem.

\begin{theoremA}\label{theoremain}
Let $(\mathfrak{L}, \cdot,  [\cdot, \cdot ])$ be a nonzero complex 
$3$-dimensional transposed Poisson algebra,
then  $(\mathfrak{L}, \cdot,  [\cdot, \cdot ])$ is isomorphic to one and only one transposed Poisson algebra
listed below:

\begin{itemize}
 
 \item ${\rm T}_{01}:$ $[e_1, e_{2}]=e_3,$  $[e_{1}, e_3]=-e_2,$  $[e_2, e_{3}]=e_1.$
 
     \item ${\rm T}_{02}:\left\{ 
\begin{tabular}{l}
$e_2 \cdot  e_2 = e_3,$\\ 
$[ e_1 , e_2 ] = e_3.$ 
\end{tabular}%
\right. $

    \item ${\rm T}_{03}^{\beta}:\left\{ 
\begin{tabular}{l}
$e_1 \cdot  e_2 = \beta e_3,$ \\ 
$[ e_1 , e_2 ] = e_3.$
\end{tabular}%
\right. $

    \item ${\rm T}_{04}^{\beta}:\left\{ 
\begin{tabular}{l}
$e_1 \cdot e_2 = \beta e_3,  e_2 \cdot e_2  = e_1,$ \\ 
$[ e_1 , e_2 ] = e_3.$
\end{tabular}%
\right. $

    \item ${\rm T}_{05}:\left\{ 
\begin{tabular}{l}
$e_1 \cdot e_1 = e_3,  e_1 \cdot e_2 = e_1,  e_2 \cdot e_2  = e_2,  e_2 \cdot e_3 = e_3,$ \\ 
$[ e_1 , e_2 ] = e_3.$
\end{tabular}%
\right. $

    \item ${\rm T}_{06}:\left\{ 
\begin{tabular}{l}
$e_1 \cdot e_2= e_1, e_2 \cdot e_2 = e_2 , e_2 \cdot e_3= e_3,$ \\ 
$[ e_1 , e_2 ] = e_3.$
\end{tabular}%
\right. $
 
 \item ${\rm T}_{07}^{\beta}:\left\{ 
\begin{tabular}{l}
$e_1 \cdot e_3 = \beta e_1$, $e_2 \cdot e_3 = \beta e_2$, $e_3 \cdot e_3 = \beta e_3,$ \\ 
$[e_1, e_{3}]=e_1,$  $[e_2, e_{3}]=e_2.$%
\end{tabular}%
\right. $
 
\item ${\rm T}_{08}:\left\{ 
\begin{tabular}{l}
$e_3 \cdot e_3 = e_1,$ \\ 
$[e_1, e_{3}]=e_1,$  $[e_2, e_{3}]=e_2.$%
\end{tabular}%
\right. $
 
  \item ${\rm T}_{09}^{\alpha, \beta}:\left\{ 
\begin{tabular}{l}
$ e_1 \cdot e_3= \beta  e_1, \ e_2 \cdot e_3 = \beta  e_2,\  e_3 \cdot e_3 = \beta e_3,$ \\ 
 $[e_1, e_{3}]=e_1+e_2,$ $[e_2, e_3]= \alpha e_2.$
\end{tabular}%
\right. $

 \item ${\rm T}_{10}^{\alpha}:\left\{ 
\begin{tabular}{l}
$ e_3 \cdot e_3 = e_2,$ \\ 
 $[e_1, e_{3}]=e_1+e_2,$ $[e_2, e_3]= \alpha e_2.$
\end{tabular}%
\right. $

 \item ${\rm T}_{11}^{\alpha}:\left\{ 
\begin{tabular}{l}
$ e_3 \cdot e_3 = e_1,$ \\ 
 $[e_1, e_{3}]=e_1+e_2,$ $[e_2, e_3]= \alpha e_2.$
\end{tabular}%
\right. $

 \item ${\rm T}_{12}^{\beta}:\left\{ 
\begin{tabular}{l}
$ e_1 \cdot e_1 =  e_2,\ e_1 \cdot e_3 =  \beta e_1, \  e_2 \cdot e_3 =   \beta e_2,  \  e_3 \cdot e_3 =   \beta  e_3,$ \\ 
 $[e_1, e_{3}]=e_1+e_2,$ $[e_2, e_3]= 2 e_2.$
\end{tabular}%
\right. $

 \item ${\rm T}_{13}:\left\{ 
\begin{tabular}{l}
$e_1 \cdot e_1 =  e_2, \  e_3 \cdot e_3 =  e_2 ,$ \\ 
 $[e_1, e_{3}]=e_1+e_2,$ $[e_2, e_3]= 2 e_2.$
\end{tabular}%
\right. $

 \item ${\rm T}_{14}:\left\{ 
\begin{tabular}{l}
$e_1 \cdot e_3 =  e_2, \ e_3 \cdot e_3 =  e_1,$ \\ 
 $[e_1, e_{3}]=e_1+e_2,$ $[e_2, e_3]= 2 e_2.$
\end{tabular}%
\right. $

 \item ${\rm T}_{15}:\left\{ 
\begin{tabular}{l}
$e_1 \cdot e_3 =  e_2 ,$ \\ 
 $[e_1, e_{3}]=e_1+e_2,$ $[e_2, e_3]= 2 e_2.$
\end{tabular}%
\right. $

 \item ${\rm T}_{16}:\left\{ 
\begin{tabular}{ll}
$ e_3 \cdot e_3= e_1+e_2,$ \\ 
 $[e_1, e_{3}]=e_1+e_2.$
\end{tabular}%
\right. $
 
  \item ${\rm T}_{17}^{\beta}:\left\{ 
\begin{tabular}{l}
$e_1 \cdot e_1=e_2,$ \
$e_1 \cdot e_2=- e_2,$ \
$e_1 \cdot e_3=\beta  e_1,$ \
$e_2 \cdot e_2=   e_2,$ \
$e_2 \cdot e_3= \beta  e_2,$ \
$e_3 \cdot e_3= \beta  e_3,$ \\ 
 $[e_1, e_{3}]=e_1+e_2$.
\end{tabular}%
\right. $

 \item ${\rm T}_{18}:\left\{ 
\begin{tabular}{l}
$e_1 \cdot e_1=e_2,$ \
$e_1 \cdot e_2=- e_2,$ \
$e_2 \cdot e_2=    e_2,$ \
$e_3 \cdot e_3= e_1+e_2,$ \\ 
 $[e_1, e_{3}]=e_1+e_2$.
\end{tabular}%
\right. $

 \item ${\rm T}_{19}^{\gamma}:\left\{ 
\begin{tabular}{l}
$e_1 \cdot e_3= \gamma e_1+ \gamma e_2,$ \
$e_3 \cdot e_3= \gamma  e_3,$ \\ 
 $[e_1, e_{3}]=e_1+e_2$.
\end{tabular}%
\right. $

 \item ${\rm T}_{20}:$ $ e_1\cdot e_1 = e_1, e_2\cdot e_2 = e_2, e_3\cdot e_3 = e_3.$

 \item ${\rm T}_{21}:$ $ e_1\cdot e_1 = e_1,   e_2\cdot e_2 = e_2, e_1\cdot e_3 = e_3.$

 \item ${\rm T}_{22}:$ $e_1 \cdot e_1 = e_1, e_2 \cdot e_2 = e_2.$ 
 
 \item ${\rm T}_{23}:$ $e_1 \cdot e_1 = e_1, e_1 \cdot e_2  =e_2, e_1\cdot e_3=e_3, e_2 \cdot e_2 = e_3.$ 

 \item ${\rm T}_{24}:$ $e_1 \cdot e_1 = e_1, e_1\cdot e_2  =e_2, e_1 \cdot e_3 =e_3.$  

 \item ${\rm T}_{25}:$ $ e_1 \cdot e_1  = e_1, e_1 \cdot e_2  =e_2.$  

 \item ${\rm T}_{26}:$ $e_1 \cdot e_1  = e_1, e_2 \cdot e_2 = e_3.$ 

 \item ${\rm T}_{27}:$ $e_1 \cdot e_1 = e_1.$ 

 \item ${\rm T}_{28}:$ $e_1 \cdot e_1  = e_2, e_1\cdot e_2 = e_3.$  

 \item ${\rm T}_{29}:$ $e_1\cdot e_2 = e_3.$ 

 \item ${\rm T}_{30}:$ $ e_1 \cdot e_1  = e_2. $

\end{itemize}

 where the parameters $\beta\in\mathbb{C}$ and $\gamma\in\mathbb{C}^*$. Between these algebras there are precisely the following non-trivial isomorphisms:
 \begin{itemize}
     \item ${\rm T}_{03}^{\beta_1}\cong {\rm T}_{03}^{\beta_2}$ if and only if $\beta_1=\pm\beta_2$.
     
     \item ${\rm T}_{09}^{\alpha_1, \beta_1}\cong {\rm T}_{09}^{\alpha_2, \beta_2}$ if and only if $(\alpha_2, \beta_2)=(\alpha_1, \beta_1)$ or  $(\alpha_2, \beta_2)=(\frac{1}{\alpha_1}, \frac{\beta_1}{\alpha_1})$.
          
     \item ${\rm T}_{11}^{\beta_1}\cong {\rm T}_{11}^{\beta_2}$ if and only if $\beta_2=\beta_1$ or $\beta_2=\frac{1}{\beta_1}$.

 \end{itemize}

\end{theoremA}

\section{Strong special transposed Poisson algebras}\label{spe3dim}

It is known that each Novikov-Poisson algebra under commutator product on non-associative multiplication gives a transposed Poisson algebra \cite{bai20}.
Let us say that a transposed Poisson algebra is special (strong special) if it can be embedded into (it is isomorphic to) a Novikov-Poisson algebra relative to the commutator bracket. 
Similarly, let us say that a transposed Poisson  algebra is  $D$-special  (strong special) (from ``differentially'') if it embeds into (it is isomorphic to) 
a commutative algebra with a derivation relative to the bracket 
$[x,y]=\mathfrak D(x)\cdot y-x\cdot \mathfrak D(y).$ Obviously, every $D$-special transposed Poisson algebra is a special one. Our main strategy for classifying all non-isomorphic strong $D$-special transposed Poisson algebras with a given associative commutative algebras $(\mathfrak A, \cdot)$ is as follows.
\begin{enumerate}
\item Find all derivations $\mathfrak{Der}({\mathfrak A}, \cdot)$; 

\item Describe the multiplication of the family of transposed Poisson algebras given by $({\mathfrak A}, \cdot, [x,y]=\mathfrak   D(x)\cdot y-x \cdot \mathfrak D(y))$; 

\item Consider the multiplication table of $({\mathfrak A}, [\cdot, \cdot] )$ under the action of elements from  ${\rm Aut}({\mathfrak A},  \cdot )$ and separate all non-isomorphic cases.
\end{enumerate}

\subsection{Strong $D$-special  $2$-dimensional transposed Poisson algebras}

Let us remember, that the classification of complex $2$-dimensional transposed Poisson algebras is given in \cite{bai20}.

\subsubsection{The algebraic classification of $2$-dimensional associative commutative algebras}

\begin{itemize} 

 \item ${\mathcal A}_{01}:$ $e_1 \cdot e_1 = e_1, e_2 \cdot e_2 = e_2.$ 

 \item ${\mathcal A}_{02}:$ $ e_1 \cdot e_1  = e_1, e_1 \cdot e_2  =e_2.$  

 \item ${\mathcal A}_{03}:$ $e_1 \cdot e_1 = e_1.$ 
 
 \item ${\mathcal A}_{04}:$ $ e_1 \cdot e_1  = e_2. $
 
 \end{itemize}
 
 By straightforward calculations we have the following result:
 
 \begin{lemma}
 
 Given $\mathcal{A} \in \left\{ {\mathcal A}_{01}, {\mathcal A}_{03} , {\mathcal A}_{04}\right\}$, then $[x,y]=\mathfrak D(x)\cdot y-x\cdot \mathfrak D(y) = 0$ for any derivation $\mathfrak D\in\mathfrak{Der}(\mathcal{A})$.
 
 \end{lemma}

\subsubsection{Strong $D$-special transposed Poisson algebras on  ${\mathcal A}_{02}$}\label{2spe}
Let $\mathfrak{D}$ be a derivation of ${\mathcal A}_{02},$
then 
\begin{center}
$\mathfrak{D}(e_2)=-\alpha e_2.$
\end{center}
Which gives 

${\mathcal D}_{01}^{\alpha}: 
\left\{ 
\begin{tabular}{l}
$e_1 \cdot e_1=e_1,$ \
$e_1 \cdot e_2=e_2,$ \\
 $[e_1,e_2]= \alpha e_2$.
\end{tabular}%
\right. $

All algebras from the family 
${\mathcal D}_{01}^{\alpha}$ are non-isomorphic and they exhaust all strong  $D$-special $2$-dimensional transposed Poisson algebras.

\subsubsection{Non-strong-$D$-special $2$-dimensional transposed Poisson algebras}
Thanks to subsubsection \ref{2spe} and \cite{bai20}, we are concluding 
that there are only two non-strong-$D$-special transposed Poisson algebras:

\begin{center}
    
${\mathcal N}_{01}: 
\left\{ 
\begin{tabular}{l}
$e_1\cdot e_1 = e_2,$\\  
$[e_1, e_2] = e_2.$

\end{tabular}%
\right. $
\text{ \ and \ } 
${\mathcal N}_{02}: 
\left\{ 
\begin{tabular}{l}
$e_1\cdot e_2 =  e_1, e_2\cdot e_2 = e_2,$ \\ 
$[e_1, e_2] = e_2.$

\end{tabular}%
\right. $

\end{center}

On the one hand, an straightforward verification shows that ${\mathcal N}_{01}$ is strong special using the Novikov-Poisson algebra:

\begin{center}
    
${\bf N}_{01}: 
\left\{ 
\begin{tabular}{l}
$e_2\cdot e_2 = e_1,$\\  
$e_2 \circ e_1 = -e_1.$

\end{tabular}%
\right. $
\end{center}

On the other hand, one can verify that the Novikov-Poisson algebras defined on $({\mathcal N}_{02}, \cdot)$ are:

\begin{center}
    
${\bf N}_{02}^{\alpha,\beta,\gamma}: 
\left\{ 
\begin{tabular}{l}
$e_1\cdot e_2 =  e_1, e_2\cdot e_2 = e_2,$\\  
$e_1 \circ e_2 = \alpha e_1, e_2 \circ e_1 = \beta e_1, e_2 \circ e_2 = \gamma e_1 + \alpha e_2.$

\end{tabular}%
\right. $
\end{center}

Then, observe that $e_1 \circ e_2 - e_2 \circ e_1 = (\alpha - \beta) e_1$ and $[e_1, e_2] = e_2$. Since the automorphisms of $({\mathcal N}_{02}, \cdot)$ verify that $\phi(e_1)=\lambda_{11}e_1,\phi(e_2) = e_2$ for $\lambda_{11}\neq0$, we conclude that ${\mathcal N}_{02}$ is not strong special.

\begin{corollary}
All complex $2$-dimensional transposed Poisson algebras are strong special.
\end{corollary}

\subsection{Strong $D$-special  $3$-dimensional transposed Poisson algebras} 

\subsubsection{The algebraic classification of $3$-dimensional associative commutative algebras} The classification of the  $3$-dimensional associative commutative algebras was extracted from  \cite{gkp21}.

\begin{itemize}
 \item ${\mathfrak A}_{01}:$ $ e_1\cdot e_1 = e_1, e_2\cdot e_2 = e_2, e_3\cdot e_3 = e_3.$

 \item ${\mathfrak A}_{02}:$ $ e_1\cdot e_1 = e_1,   e_2\cdot e_2 = e_2, e_1\cdot e_3 = e_3.$

 \item ${\mathfrak A}_{03}:$ $e_1 \cdot e_1 = e_1, e_2 \cdot e_2 = e_2.$ 
 
 \item ${\mathfrak A}_{04}:$ $e_1 \cdot e_1 = e_1, e_1 \cdot e_2  =e_2, e_1\cdot e_3=e_3, e_2 \cdot e_2 = e_3.$ 

 \item ${\mathfrak A}_{05}:$ $e_1 \cdot e_1 = e_1, e_1\cdot e_2  =e_2, e_1 \cdot e_3 =e_3.$  

 \item ${\mathfrak A}_{06}:$ $ e_1 \cdot e_1  = e_1, e_1 \cdot e_2  =e_2.$  

 \item ${\mathfrak A}_{07}:$ $e_1 \cdot e_1  = e_1, e_2 \cdot e_2 = e_3.$ 

 \item ${\mathfrak A}_{08}:$ $e_1 \cdot e_1 = e_1.$ 

 \item ${\mathfrak A}_{09}:$ $e_1 \cdot e_1  = e_2, e_1\cdot e_2 = e_3.$  

 \item ${\mathfrak A}_{10}:$ $e_1\cdot e_2 = e_3.$ 

 \item ${\mathfrak A}_{11}:$ $ e_1 \cdot e_1  = e_2. $
 
 \end{itemize}
 
 By straightforward calculations we have the following result:
 
 \begin{lemma}
 
 Given $\mathfrak{A} \in \left\{ {\mathfrak A}_{01}, {\mathfrak A}_{03}, {\mathfrak A}_{07}, {\mathfrak A}_{08}, {\mathfrak A}_{11}\right\}$, then 
 $[x,y]=\mathfrak D(x)\cdot y-x\cdot \mathfrak D(y) = 0$ for any derivation $\mathfrak D\in\mathfrak{Der}(\mathfrak{A})$.
 
 \end{lemma}

\subsubsection{Strong $D$-special transposed Poisson algebras on  ${\mathfrak A}_{02}$}
Let $\mathfrak{D}$ be a derivation of ${\mathfrak A}_{02},$
then 
\begin{center}
$\mathfrak{D}(e_3)=-\alpha e_3.$
\end{center}
Which gives 

${\rm A}_{01}^{\alpha}: 
\left\{ 
\begin{tabular}{l}
$e_1 \cdot e_1=e_1,$ \
$e_2 \cdot e_2=e_2,$ \
$e_1 \cdot e_3=    e_3,$  \\ 
 $[e_1,e_3]= \alpha e_3$.
\end{tabular}%
\right. $

All algebras from the family 
${\rm A}_{01}^{\alpha}$ are non-isomorphic and under the following change of basis:
\begin{center}
    $E_1=e_3-e_2,$ \ $E_2=e_2,$ \ $E_3=-\alpha^{-1}(e_1+e_2),$
\end{center}
we have that  ${\rm A}_{01}^{\alpha} \cong {\rm T}_{17}^{-\alpha^{-1}}$ for $\alpha\neq0$. 

\subsubsection{Strong $D$-special transposed Poisson algebras on  ${\mathfrak A}_{04}$}
Let $\mathfrak{D}$ be a derivation of ${\mathfrak A}_{04},$
then 
\begin{center}
$\mathfrak{D}(e_2)= -\alpha e_2  -\beta e_3,$
$\mathfrak{D}(e_3)= -2 \alpha e_3.$
\end{center}
Which gives 

${\rm A}_{02}^{\alpha, \beta}: 
\left\{ 
\begin{tabular}{l}
$e_1 \cdot e_1 = e_1, e_1 \cdot e_2  =e_2, e_1\cdot e_3=e_3, e_2 \cdot e_2 = e_3,$   \\ 
 $[e_1,e_2]= \alpha e_2 + \beta e_3, [e_1,e_3]= 2 \alpha e_3$.
\end{tabular}%
\right. $

Let us consider two separate cases:

\begin{enumerate}
    \item[$\bullet$] $\alpha=0,$ then
    under  the following changing 
  \begin{center}
    $E_1=-\beta e_2,$ \ $E_2=e_1,$ \ $E_3=-\beta^2 e_3,$
\end{center}
we have that  ${\rm A}_{02}^{0,\beta} \cong {\rm T}_{05}$ for $\beta\neq0$. 

   \item[$\bullet$] $\alpha\neq 0,$ then
    under  the following changing 
  \begin{center}
    $E_1= e_2+(1-\beta\alpha^{-1})e_3,$ \ $E_2=e_3,$ \ $E_3=-\alpha^{-1} e_1,$
\end{center}
we have that  ${\rm A}_{02}^{\alpha,\beta} \cong {\rm T}_{12}^{-\alpha^{-1}}.$
\end{enumerate}

\subsubsection{Strong $D$-special transposed Poisson algebras on  ${\mathfrak A}_{05}$}
Let $\mathfrak{D}$ be a derivation of ${\mathfrak A}_{05},$
then 
\begin{center}
$\mathfrak{D}(e_2)=-\alpha e_2 - \beta e_3,$ $\mathfrak{D}(e_3)= -\gamma e_2 - \delta e_3.$
\end{center}
Which gives 

${\rm A}_{03}^{\alpha, \beta, \gamma, \delta}: 
\left\{ 
\begin{tabular}{l}
$e_1 \cdot e_1 = e_1, e_1\cdot e_2  =e_2, e_1 \cdot e_3 =e_3,$  \\ 
 $[e_1,e_2]= \alpha e_2 + \beta e_3, [e_1,e_3]= \gamma e_2 + \delta e_3$.
\end{tabular}%
\right. $

It is easy to see that if $({\rm T}_j, \cdot) \cong {\mathfrak A}_{05}$ then $j=06, 07,09.$ Obtaining the following mutually exclusive cases for suitable parameters $\alpha'$ and $\beta'$:
$${\rm T}_{06} \cong {\rm A}_{03}^{\alpha, \beta, -\frac{\alpha^2}{\beta}, -\alpha}, \ {\rm T}_{07}^{\alpha} \cong {\rm A}_{03}^{-\alpha^{-1}, 0, 0, -\alpha^{-1}}, {\rm T}_{09}^{\alpha',\beta'} \cong {\rm A}_{03}^{\alpha, \beta, \gamma, \delta}.$$

%
%
%

\subsubsection{Strong $D$-special transposed Poisson algebras on  ${\mathfrak A}_{06}$}
Let $\mathfrak{D}$ be a derivation of ${\mathfrak A}_{06},$
then 
\begin{center}
$\mathfrak{D}(e_2)= - \alpha e_2, \mathfrak{D}(e_3)= \beta e_3.$
\end{center}
Which gives 

${\rm A}_{04}^{\alpha}: 
\left\{ 
\begin{tabular}{l}
$ e_1 \cdot e_1  = e_1, e_1 \cdot e_2  =e_2,$  \\ 
 $[e_1,e_2]= \alpha e_2 $.
\end{tabular}%
\right. $

All algebras from the family 
${\rm A}_{04}^{\alpha}$ are non-isomorphic and under the following changing
\begin{center}
    $E_1=e_2-e_3,$ \ $E_2=e_3,$ \ $E_3=-\alpha^{-1}e_1,$
\end{center}
we have that  ${\rm A}_{04}^{\alpha} \cong {\rm T}_{19}^{-\alpha^{-1}}$ for $\alpha\neq0$.

\subsubsection{Strong $D$-special transposed Poisson algebras on  ${\mathfrak A}_{09}$}
Let $\mathfrak{D}$ be a derivation of ${\mathfrak A}_{09},$
then 
\begin{center}
$\mathfrak{D}(e_1)= -\alpha e_1 + \beta e_2 + \gamma e_3$, $\mathfrak{D}(e_2)= -2 \alpha e_2 + 2\beta e_3$, $\mathfrak{D}(e_3)= -3\alpha e_3$, 
\end{center}
Which gives 

${\rm A}_{05}^{\alpha}: 
\left\{ 
\begin{tabular}{l}
$e_1 \cdot e_1  = e_2, e_1\cdot e_2 = e_3,$  \\ 
 $[e_1,e_2]= \alpha e_3$.
\end{tabular}%
\right. $

All algebras from the family 
${\rm A}_{05}^{\alpha}$ are non-isomorphic and under the following changing
\begin{center}
    $E_1=e_2,$ \ $E_2=e_1,$ \ $E_3= - \alpha e_3,$
\end{center}
we have that  ${\rm A}_{05}^{\alpha} \cong {\rm T}_{04}^{-\alpha^{-1}}$ for $\alpha\neq0$.

\subsubsection{Strong $D$-special transposed Poisson algebras on  ${\mathfrak A}_{10}$}
Let $\mathfrak{D}$ be a derivation of ${\mathfrak A}_{10},$
then 
\begin{center}
$\mathfrak{D}(e_3)=\alpha e_1 + \delta e_3,$ $\mathfrak{D}(e_2)= \beta e_2 + \gamma e_3,$ $\mathfrak{D}(e_3)=(\alpha+\beta) e_3.$ 
\end{center}
Denoting $\epsilon = \alpha - \beta$, we have

${\rm A}_{06}^{\epsilon}: 
\left\{ 
\begin{tabular}{l}
$e_1\cdot e_2 = e_3,$  \\ 
 $[e_1,e_2]= \epsilon e_3 $.
\end{tabular}%
\right. $

All algebras from the family 
${\rm A}_{06}^{\epsilon}$ are isomorphic only on the case
${\rm A}_{06}^{\epsilon}\cong {\rm A}_{06}^{-\epsilon}$ 
and under the following change of basis
\begin{center}
    $E_1=e_2,$ \ $E_2=e_1,$ \ $E_3= - \epsilon e_3,$
\end{center}
we have that  ${\rm A}_{06}^{\epsilon} \cong {\rm T}_{03}^{-\epsilon^{-1}}$ for $\epsilon\neq0$.

\subsection{Classification theorem}
 The results of the previous subsection together with the classification of the commutative associative algebras of dimension three (we recall the classification that was used in \cite{gkp21}), give us the following classification theorem.

\begin{theoremB} 
Let $(\mathfrak{L}, \cdot,  [\cdot, \cdot ])$ be a nontrivial complex 
$3$-dimensional transposed Poisson algebra (i.e. $\cdot\neq 0$ and $[\cdot, \cdot ] \neq 0$).
If $(\mathfrak{L}, \cdot,  [\cdot, \cdot ])$ is a strong $D$-special transposed Poisson algebra
then it is isomorphic to one and only one listed below:
 
\begin{longtable}{lllll}

${\rm D}_{01}^{\alpha}$ & $:$ & 
${\rm A}_{01}^{\alpha}$ & $:$ & 
$\left\{ 
\begin{tabular}{l}
$e_1 \cdot e_1=e_1,$ \
$e_2 \cdot e_2=e_2,$ \
$e_1 \cdot e_3=    e_3,$  \\ 
 $[e_1,e_3]= \alpha e_3$.
\end{tabular}
\right. $\\

${\rm D}_{02}$ & $:$ & 
${\rm A}_{02}^{0,1}$&$:$ & 
$\left\{ 
\begin{tabular}{l}
$e_1 \cdot e_1 = e_1, e_1 \cdot e_2  =e_2, e_1\cdot e_3=e_3, e_2 \cdot e_2 = e_3,$   \\ 
 $[e_1,e_2]=   e_3 $.
\end{tabular}%
\right. $\\

${\rm D}_{03}^{\alpha}$ & $:$ & 
${\rm A}_{02}^{\alpha,0}$&$:$ & 
$\left\{ 
\begin{tabular}{l}
$e_1 \cdot e_1 = e_1, e_1 \cdot e_2  =e_2, e_1\cdot e_3=e_3, e_2 \cdot e_2 = e_3,$   \\ 
 $[e_1,e_2]=   e_2 , [e_1,e_3]= 2  e_3$.
\end{tabular}%
\right. $\\

${\rm D}_{04}$ & $:$ & 
${\rm A}_{03}^{0,1,0,0}$&$:$ & 
$\left\{ 
\begin{tabular}{l}
$e_1 \cdot e_1 = e_1, e_1\cdot e_2  =e_2, e_1 \cdot e_3 =e_3,$  \\ 
 $[e_1,e_2]=  e_3 $.
\end{tabular}%
\right. $\\

${\rm D}_{05}^{\alpha}$ & $:$ & 
${\rm A}_{03}^{\alpha,0,0,\alpha}$&$:$ & 
$\left\{ 
\begin{tabular}{l}
$e_1 \cdot e_1 = e_1, e_1\cdot e_2  =e_2, e_1 \cdot e_3 =e_3,$  \\ 
 $[e_1,e_2]= \alpha e_2 , [e_1,e_3]= \alpha   e_3$.
\end{tabular}%
\right. $\\

${\rm D}_{06}^{\alpha,\beta}$ & $:$ & 
${\rm A}_{03}^{\alpha,0,\beta,\beta}$&$:$ & 
$\left\{ 
\begin{tabular}{l}
$e_1 \cdot e_1 = e_1, e_1\cdot e_2  =e_2, e_1 \cdot e_3 =e_3,$  \\ 
 $[e_1,e_2]= \alpha  e_2 , [e_1,e_3]= \beta e_2 + \beta e_3$.
\end{tabular}%
\right. $\\

${\rm D}_{06}^{\alpha}$ & $:$ & 
${\rm A}_{04}^{\alpha}$&$:$ & 
$\left\{ 
\begin{tabular}{l}
$ e_1 \cdot e_1  = e_1, e_1 \cdot e_2  =e_2,$  \\ 
 $[e_1,e_2]= \alpha e_2 $.
\end{tabular}%
\right. $\\

${\rm D}_{07}^{\alpha}$ & $:$ & 
${\rm A}_{05}^{\alpha}$&$:$ & 
$\left\{ 
\begin{tabular}{l}
$e_1 \cdot e_1  = e_2, e_1\cdot e_2 = e_3,$  \\ 
 $[e_1,e_2]= \alpha e_3$.
\end{tabular}%
\right. $\\

${\rm D}_{08}^{\alpha}$ & $:$ & 
${\rm A}_{06}^{\alpha}$&$:$ & 
$\left\{ 
\begin{tabular}{l}
$e_1\cdot e_2 = e_3,$  \\ 
 $[e_1,e_2]= \alpha e_3 $.
\end{tabular}%
\right.$
\end{longtable}
 Where 
\begin{center}
${\rm D}_{06}^{\alpha, \beta} \cong {\rm D}_{06}^{\beta, \alpha}$ \ and \  
${\rm D}_{08}^{\alpha} \cong {\rm D}_{08}^{-\alpha}.$

\end{center} 

If $(\mathfrak{L}, \cdot,  [\cdot, \cdot ])$ is a non-strong-$D$-special transposed Poisson algebra
then it is isomorphic to one and only one listed below
(see, also Theorem A):
\begin{center}
    
${\rm T}_{02}, \ {\rm T}^{\alpha\neq 0}_{03}, \ {\rm T}_{08}, \ {\rm T}_{10}^{\alpha},  \ 
{\rm T}_{11}^{\alpha},  \ {\rm T}_{13},  
\ {\rm T}_{14}, \ {\rm T}_{15}, \ {\rm T}_{16},  \ {\rm T}_{18}.$

\end{center}

\end{theoremB}

\section{The geometric classification of transposed Poisson algebras}\label{geo}

Given a vector space ${\mathbb V}$ of dimension $n$, the set of bilinear maps \begin{center}$\textrm{Bil}({\mathbb V} \times {\mathbb V}, {\mathbb V}) \cong \textrm{Hom}({\mathbb V} ^{\otimes2}, {\mathbb V})\cong ({\mathbb V}^*)^{\otimes2} \otimes {\mathbb V}$
\end{center} is a vector space of dimension $n^3$. The set of pairs of bilinear maps (or bilinear pairs) 
\begin{center}$\textrm{Bil}({\mathbb V} \times {\mathbb V}, {\mathbb V}) \oplus \textrm{Bil}({\mathbb V}\times {\mathbb V}, {\mathbb V}) \cong ({\mathbb V}^*)^{\otimes2} \otimes {\mathbb V} \oplus ({\mathbb V}^*)^{\otimes2} \otimes {\mathbb V}$ \end{center} which is a vector space of dimension $2n^3$. This vector space has the structure of the affine space $\mathbb{C}^{2n^3}$ in the following sense:
fixed a basis $e_1, \ldots, e_n$ of ${\mathbb V}$, then any pair with multiplication $(\mu, \mu')$, is determined by some parameters $c_{ij}^k, c_{ij}'^k \in \mathbb{C}$,  called {structural constants},  such that
$$\mu(e_i, e_j) = \sum_{p=1}^n c_{ij}^k e_k \textrm{ and } \mu'(e_i, e_j) = \sum_{p=1}^n c_{ij}'^k e_k$$
which corresponds to a point in the affine space $\mathbb{C}^{2n^3}$. Then a set of bilinear pairs $\mathcal S$ corresponds to an algebraic variety, i.e., a Zariski closed set if there are some polynomial equations in variables $c_{ij}^k, c_{ij}'^k$ with zero locus equal to the set of structural constants of the bilinear pairs in $\mathcal S$. Since given the identities defining transposed Poisson algebras we can obtain a set of polynomial equations in variables $c_{ij}^k, c_{ij}'^k$, the class of $n$-dimensional transposed Poisson algebras is a variety, denote it by $\mathcal{T}_{n}$.
Now, consider the following action of $\textrm{GL}({\mathbb V})$ on $\mathcal{T}_{n}$:
$$(g*(\mu, \mu'))(x,y) := (g \mu (g^{-1} x, g^{-1} y), g \mu' (g^{-1} x, g^{-1} y))$$
for $g\in\textrm{GL}({\mathbb V})$, $(\mu, \mu')\in \mathcal{T}_{n}$ and for any $x, y \in {\mathbb V}$. Observe that the $\textrm{GL}({\mathbb V})$-orbit of $(\mu, \mu')$, denoted $O((\mu, \mu'))$, contains all the structural constants of the bilinear pairs isomorphic to the transposed Poisson algebra with structural constants $(\mu, \mu')$.

The purpose of this section is to describe the irreducible components of the variety ${\mathcal{T}_{3}}$. Recall that any affine variety can be represented as a finite union of its irreducible components in a unique way.
Additionally, describing the irreducible components of  ${\mathcal{T}_{3}}$ gives us the rigid bilinear pairs of the variety, which are those bilinear pairs with an open $\textrm{GL}(\mathbb V)$-orbit. This is because a bilinear pair is rigid in a variety if and only if the closure of its orbit is an irreducible component of the variety.
For this, we have to introduce the following notion. Denote by $\overline{O((\mu, \mu'))}$ the closure of the orbit of $(\mu, \mu')\in{\mathcal{T}_{n}}$.

\begin{definition}
\rm Let ${\rm T} $ and ${\rm T}'$ be two $n$-dimensional transposed Poisson algebras and $(\mu, \mu'), (\lambda,\lambda') \in \mathcal{T}_{n}$ be their representatives in the affine space, respectively. We say ${\rm T}$ {degenerates}  to ${\rm T}'$, and write ${\rm T} \to {\rm T} '$, if $(\lambda,\lambda')\in\overline{O((\mu, \mu'))}$. If ${\rm T}  \not\cong {\rm T}'$, then we call it a  {proper degeneration}.

Conversely, if $(\lambda,\lambda')\not\in\overline{O((\mu, \mu'))}$ then we call it a 
{non-degeneration} and we write ${{\rm T} }\not\to {{\rm T} }'$.
\end{definition}

Furthermore, for a parametric family of transposed Poisson algebras we have the following notion.

\begin{definition}
\rm
Let ${{\rm T} }(*)=\{{{\rm T} }(\alpha): {\alpha\in I}\}$ be a family of $n$-dimensional transposed Poisson algebras and let ${{\rm T} }'$ be another transposed Poisson algebra. Suppose that ${{\rm T} }(\alpha)$ is represented by the structure $(\mu(\alpha),\mu'(\alpha))\in{\mathcal{T} }_n$ for $\alpha\in I$ and ${{\rm T} }'$ is represented by the structure $(\lambda, \lambda')\in{\mathcal{T} }_n$. We say the family ${{\rm T} }(*)$ {degenerates}   to ${{\rm T} }'$, and write ${{\rm T} }(*)\to {{\rm T} }'$, if $(\lambda,\lambda')\in\overline{\{O((\mu(\alpha),\mu'(\alpha)))\}_{\alpha\in I}}$.

Conversely, if $(\lambda,\lambda')\not\in\overline{\{O((\mu(\alpha),\mu'(\alpha)))\}_{\alpha\in I}}$ then we call it a  {non-degeneration}, and we write ${{\rm T} }(*)\not\to {{\rm T} }'$.

\end{definition}

Observe that ${\rm T}'$ corresponds to an irreducible component of $\mathcal{T}_3$ (more precisely, $\overline{{\rm T}'}$ is an irreducible component) if and only if ${{\rm T} }\not\to {{\rm T} }'$ for any 3-dimensional transposed Poisson algebra ${\rm T}$ and ${{\rm T}(*) }\not\to {{\rm T} }'$ for any parametric family of 3-dimensional transposed Poisson algebras ${\rm T}(*)$. To prove this, we will use the next ideas.

Firstly, since $\mathrm{dim}\,O((\mu, \mu')) = n^2 - \mathrm{dim}\,\mathfrak{Der}({\rm T})$, then if $ {\rm T} \to  {\rm T} '$ and  ${\rm T} \not\cong  {\rm T} '$, we have that $\mathrm{dim}\,\mathfrak{Der}( {\rm T} )<\mathrm{dim}\,\mathfrak{Der}( {\rm T} ')$, where $\mathfrak{Der}( {\rm T} )$ denotes the Lie algebra of derivations of  ${\rm T} $. 

{
Secondly, let ${\rm T}$ and ${\rm T}'$ be two transposed Poisson algebras represented by the structures $(\mu, \mu')$ and $(\lambda, \lambda')$ from $\mathcal{T}_n$, respectively. If there exist a parametrized change of basis $g: \mathbb{C}^* \to \textrm{GL}({\mathbb V})$ such that:
$$\lim\limits_{t\to 0} g(t)*(\mu, \mu') = (\lambda, \lambda'),$$
then ${\rm T}\to {\rm T}'$. To prove primary degenerations, we will provide the map $g$.}

Thirdly, to prove non-degenerations we may use a remark that follows from this lemma (see \cite{afm}). 

\begin{lemma}\label{main1}
Consider two transposed Poisson algebras ${\rm T}$ and ${\rm T}'$. Suppose ${\rm T} \to {\rm T}'$. Let C be a Zariski closed in ${\mathcal T}_n$ that is stable by the action of the invertible upper (lower) triangular matrices. Then if there is a representation $(\mu, \mu')$ of ${\rm T}$ in C, then there is a representation $(\lambda, \lambda')$ of ${\rm T}'$ in C.
\end{lemma}

\begin{remark}
\label{redbil}

Moreover, let ${{\rm T} }$ and ${{\rm T} }'$ be two transposed Poisson algebras represented by the structures $(\mu, \mu')$ and $(\lambda, \lambda')$ from ${{\mathcal T} }_n$. Suppose ${\rm T}\to{\rm T}'$. Then if $\mu, \mu', \lambda, \lambda'$ represent  algebras ${\rm T}_{0}, {\rm T}_{1}, {\rm T}'_{0}, {\rm T}'_{1}$ in the affine space $\mathbb{C}^{n^3}$ of algebras with a single multiplication, respectively, we have ${\rm T}_{0}\to {\rm T}'_{0}$ and $ {\rm T}_{1}\to {\rm T}'_{1}$.
So for example, $(0, \mu)$ can not degenerate in $(\lambda, 0)$ unless $\lambda=0$. 

\end{remark}

Fourthly, to prove ${\rm T}(*)\to {\rm T}'$, suppose that ${\rm T}(\alpha)$ is represented by the structure $(\mu(\alpha),\mu'(\alpha))\in\mathcal{T}_{n}$ for $\alpha\in I$ and ${\rm T}'$ is represented by the structure $(\lambda, \lambda')\in\mathcal{T}_{n}$. If there exists a pair of maps $(f, g)$, where $f:\mathbb{C}^*\to I$ and $g: \mathbb{C}^* \to \textrm{GL}({\mathbb V})$ are such that:
$$\lim\limits_{t\to 0} g(t)*(\mu\big(f(t)\big),\mu'\big(f(t)\big)) = (\lambda, \lambda'),$$
then ${\rm T}(*)\to {\rm T}'$.

Lastly, to prove ${{\rm T} }(*)\not \to {{\rm T} }'$, we may use an analogous of Remark \ref{redbil} for parametric families that follows from Lemma \ref{main2}.

\begin{lemma}\label{main2}
Consider the family of transposed Poisson algebras ${\rm T}(*)$ and the transposed Poisson algebra ${\rm T}'$. Suppose ${\rm T}(*) \to {\rm T}'$. Let C be a Zariski closed in $\mathcal{T}_n$ that is stable by the action of the invertible upper (lower) triangular matrices. Then if there is a representation $(\mu(\alpha), \mu'(\alpha))$ of ${\rm T}(\alpha)$ in C for every $\alpha\in I$, then there is a representation $(\lambda, \lambda')$ of ${\rm T}'$ in C.
\end{lemma}

The following result by \cite{gkp21}  will be used. 

\begin{theorem}
The variety of $3$-dimensional commutative associative algebras has a single irreducible component corresponding to the algebra $({\rm T}_{20}, \cdot)$. Moreover, $\textrm{dim}(\overline{({\rm T}_{20}, \cdot)})=9$.
\end{theorem}

By this theorem, to find the irreducible components of the variety of $3$-dimensional transposed Poisson algebras we only have to study the degenerations and non-degenerations between the algebras ${\rm T}_{i}$, with $i=01, \ldots, 20$. The geometric classification of the variety $\mathcal{T}_3$
 is given in Theorem \ref{geom}.

\begin{theoremC}\label{geom}
The variety of $3$-dimensional transposed Poisson algebras has five irreducible components corresponding to the rigid algebras ${\rm T}_{01}$ and ${\rm T}_{20}$
 and the parametric families ${\rm T}_{09}^{\alpha,\beta}$, ${\rm T}_{12}^\beta$ and ${\rm T}_{17}^{\beta}$.
\end{theoremC}
\begin{proof}

Our strategy to proof the result consists in first showing that every transposed Poisson can be obtained through a degeneration of one algebra from one of the five irreducible components proposed. These degenerations follow from the table below. Then from the orbit dimensions, we just miss the non-degenerations between ${\rm T}\in \left\{{\rm T}_{09}^{\alpha,\beta}, {\rm T}_{12}^\beta, {\rm T}_{17}^{\beta}, {\rm T}_{20}\right\}$ and ${\rm T}_{01}$. Since $\mathfrak{sl}_2$ is an irreducible component of the variety of $3$-dimensional Lie algebras and ${\rm T}_{01}$ is the only algebra such that $({\rm T}_{01}, [-,-])\cong \mathfrak{sl}_2$, then ${\rm T}_{01}$ is an irreducible component.

\end{proof}

{\tiny
\begin{longtable}{|lcl|lll|}
\hline
\multicolumn{3}{|c|}{\textrm{Degeneration}}  & \multicolumn{3}{|c|}{\textrm{Parametrized basis}} \\
\hline
\hline

${\rm T}_{05} $ & $\to$ & ${\rm T}_{02}$ & $
g_{1}(t)= t^{-4}e_1,$ & $
g_{2}(t)= -t^{-3}e_1 + t^{-2}e_2,$ & $
g_{3}(t)= t^{-6}e_3.$ \\
\hline

${\rm T}_{04}^{\alpha} $ & $\to$ & ${\rm T}_{03}^{\alpha}$ & $
g_{1}(t)= e_1,$ & $
g_{2}(t)= t^{-1}e_2,$ & $
g_{3}(t)= t^{-1}e_3.$ \\
\hline

${\rm T}_{05} $ & $\to$ & ${\rm T}_{04}^{\beta\neq0}$ & $
g_{1}(t)= t^{-2}e_1,$ & $
g_{2}(t)= -\beta t^{-2}e_1 + \sqrt{\beta} t^{-1}e_2 + \sqrt{\beta^5}t^{-3}e_3,$ & $
g_{3}(t)= \sqrt{\beta}t^{-3}e_3.$ \\
\hline

${\rm T}_{05} $ & $\to$ & ${\rm T}_{04}^{0}$ & $
g_{1}(t)=  t^{-4} e_1 ,$ & $
g_{2}(t)=  - t^{-2}e_1 + t^{-1}e_2,$ & $
g_{3}(t)=  t^{-5}e_3.$ \\
\hline

${\rm T}_{17}^{\frac{1}{t}} $ & $\to$ & ${\rm T}_{05}$ & $
g_{1}(t)= t^{-1}e_1,$ & $
g_{2}(t)= \frac{1}{2}e_2 + t^{-2}e_3,$ & $
g_{3}(t)= t^{-1}e_2.$ \\
\hline

${\rm T}_{05} $ & $\to$ & ${\rm T}_{06}$ & $
g_{1}(t)= t^{-1}e_1,$ & $
g_{2}(t)= e_2,$ & $
g_{3}(t)= t^{-1}e_3.$ \\
\hline

${\rm T}_{09}^{1, \beta} $ & $\to$ & ${\rm T}_{07}^{\beta}$ & $
g_{1}(t)= e_1,$ & $
g_{2}(t)= t e_2,$ & $
g_{3}(t)= e_3.$ \\
\hline

${\rm T}_{10}^1 $ & $\to$ & ${\rm T}_{08}$ & $
g_{1}(t)= 2 t^{-2}e_1,$ & $
g_{2}(t)= e_1 + t e_2,$ & $
g_{3}(t)= e_3.$ \\
\hline

${\rm T}_{11}^{\alpha} $ & $\to$ & ${\rm T}_{10}^{\alpha}$ & $
g_{1}(t)= t e_1 + e_2,$ & $
g_{2}(t)= (\alpha + t - 1)e_2,$ & $
g_{3}(t)= e_3.$ \\
\hline

${\rm T}_{09}^{t} $ & $\to$ & ${\rm T}_{11}^{\alpha}$ & $
g_{1}(t)= e_1,$ & $
g_{2}(t)= e_2,$ & $
g_{3}(t)= -t^{-1}e_1 + e_3.$ \\
\hline

${\rm T}_{12}^{t} $ & $\to$ & ${\rm T}_{13}$ & $
g_{1}(t)= e_1,$ & $
g_{2}(t)= e_2,$ & $
g_{3}(t)= -t^{-1}e_2 + e_3.$ \\
\hline

${\rm T}_{12}^{t} $ & $\to$ & ${\rm T}_{14}$ & $
g_{1}(t)= t^{-1}e_1,$ & $
g_{2}(t)= t^{-1}e_2,$ & $
g_{3}(t)= -t^{-1} + t^{-2}e_2 + e_3.$ \\
\hline

${\rm T}_{14} $ & $\to$ & ${\rm T}_{15}$ & $
g_{1}(t)= te_1,$ & $
g_{2}(t)= te_2,$ & $
g_{3}(t)= e_3.$ \\
\hline

${\rm T}_{11}^{\alpha} $ & $\to$ & ${\rm T}_{16}$ & $
g_{1}(t)= (t-\alpha+1)e_1 + e_2,$ & $
g_{2}(t)= t e_2,$ & $
g_{3}(t)= e_3.$ \\
\hline

${\rm T}_{18} $ & $\to$ & ${\rm T}_{17}^0$ & $
g_{1}(t)= t e_1 + (t-1)e_2,$ & $
g_{2}(t)= e_2,$ & $
g_{3}(t)= e_3.$ \\
\hline

${\rm T}_{17}^t $ & $\to$ & ${\rm T}_{18}$ & $
g_{1}(t)= e_1,$ & $
g_{2}(t)= e_2,$ & $
g_{3}(t)= -t^{-1} - t^{-1}e_2 + e_3.$ \\
\hline

${\rm T}_{17}^{\gamma} $ & $\to$ & ${\rm T}_{19}^{\gamma}$ & $
g_{1}(t)= t^{-1}e_1,$ & $
g_{2}(t)= t^{-1}e_2,$ & $
g_{3}(t)= \gamma t^{-1}e_2 + e_3.$ \\
\hline

\end{longtable}}


\begin{thebibliography}{99}
  
 
 
 

 
 


\bibitem{afm}
 Abdelwahab H., 
 Fernández Ouaridi A., 
 Martín González C.,  
 Degenerations of Poisson algebras, 
 Journal of Algebra and Its Applications, 2023, 
 DOI:10.1142/S0219498825500872.
 
 
\bibitem{ak21}
  Alvarez M.A.,  Kaygorodov I.,  
  The algebraic and geometric classification of nilpotent weakly associative and symmetric Leibniz algebras,  
 Journal of  Algebra, 588  (2021),  278--314.
 
 
\bibitem{afk21}
Alvarez M.,   Fehlberg J\'{u}nior  R.,  Kaygorodov I.,  
 The algebraic and geometric classification of Zinbiel algebras, 
Journal of Pure and Applied Algebra,   226 (2022), 11,  107106.
 

\bibitem{bai20}
 Bai C., Bai R., Guo L., Wu Y.,
Transposed Poisson algebras, Novikov-Poisson algebras, and 3-Lie algebras, Journal of Algebra, 632 (2023), 535--566.

 
 
\bibitem{bfk22}
 Beites P. D., 
 Ferreira B. L. M., 
  Kaygorodov I.,  
 Transposed Poisson   structures, arXiv:2207.00281
 


 
 \bibitem{blsm17}
 Bell J., Launois S.,  Sánchez O.,  Moosa R., 
Poisson algebras via model theory and differential-algebraic geometry, 
Journal of the European Mathematical Society (JEMS), 19 (2017), 7, 2019--2049.


 
 
\bibitem{ccsmv}	
 Cabrera Casado Yo., Siles Molina M.,  Velasco M., 
 Classification of three-dimensional evolution algebras, 
 Linear Algebra and Its Applications, 524 (2017), 68--108.
 
 \bibitem{ckls20}
Camacho L., Kaygorodov I.,  Lopatkin V., Salim M., 
The variety of dual Mock-Lie algebras,  
Communications in Mathematics,   28 (2020), 2,  161--178.   
 

 \bibitem{cibils}  
Cibils C., 
    $2$-nilpotent and rigid finite-dimensional algebras,
    Journal of the London Mathematical Society (2), 36 (1987), 2, 211--218. 
 
 
 \bibitem{chouhy}
Chouhy S.,
    On geometric degenerations and Gerstenhaber formal deformations,
    Bulletin of the London Mathematical Society, 51 (2019),  5, 787--797.
    
 \bibitem{erik}
Darpö E.,  Rochdi A.,  
    Classification of the four-dimensional power-commutative real division algebras, 
    Proceedings of the Royal Society of Edinburgh, Section A, 141 (2011), 6, 1207--1223.

   \bibitem{Ernst}
Dieterich E., Öhman J., 
On the classification of $4$-dimensional quadratic division algebras over square-ordered fields,  
Journal of the London Mathematical Society (2), 65 (2002), 2, 285--302.
 
 \bibitem{d19}
 Dotsenko V., 
 Algebraic structures of $F$-manifolds via pre-Lie algebras,  
 Annali di Matematica Pura ed Applicata, 198 (2019), 2, 517--527.
 
 
 \bibitem{FK21}
 Fehlberg Júnior R.,  Kaygorodov I.,
  On the Kantor product, II,
  Carpathian Mathematical Publications, 14 (2022),  2, 543--563.

  
  \bibitem{fkk21}
    Fehlberg J\'{u}nior  R.,  Kaygorodov I., Kuster C.,  
 The algebraic and geometric classification of antiassociative algebras, 
   Revista de la Real Academia de Ciencias Exactas, Físicas y Naturales. Serie A. Matemáticas,   116  (2022), 2,   78.
   
   
   \bibitem{fkkv22}
Fern\'andez Ouaridi A., Kaygorodov I., Khrypchenko M.,  Volkov Yu., 
    Degenerations of nilpotent  algebras, 
    Journal of Pure and Applied Algebra,   226 (2022),  3, 106850.
 
 
\bibitem{FKL}
Ferreira B. L. M., Kaygorodov I., Lopatkin V., $\frac{1}{2}$-derivations of Lie algebras and transposed Poisson algebras, 
Revista de la Real Academia de Ciencias Exactas, Físicas y Naturales. Serie A. Matemáticas, 115 (2021), 3, 142. 
 
 
\bibitem{fil1} Filippov V.,
 $\delta$-Derivations of Lie algebras,
Siberian Mathematical Journal, 39 (1998), 6, 1218--1230.
 
 
 \bibitem{gabriel}
Gabriel P.,
Finite representation type is open,
Proceedings of the International Conference on Representations of Algebras (Carleton Univ., Ottawa, Ont., 1974), pp. 132--155.

    \bibitem{ger63}
Gerstenhaber M.,
    On the deformation of rings and algebras,
    Annals of Mathematics (2), 79 (1964), 59--103.

\bibitem{gorb93} 
Gorbatsevich V., 
    Anticommutative finite-dimensional algebras of the first three levels of complexity, 
    St. Petersburg Mathematical Journal, 5 (1994), 3, 505--521.
 

\bibitem{gkp21}
Gorshkov I., Kaygorodov I., Popov Yu., 
Degenerations of Jordan algebras and "Marginal'' algebras,  
Algebra Colloquium, 28 (2021), 2, 281--294. 


  
 
   
 
    \bibitem{GRH} 
Grunewald F.,  O'Halloran J.,
    Varieties of nilpotent Lie algebras of dimension less than six,
    Journal of Algebra, 112 (1988), 2, 315--325.
  
 
\bibitem{jkk21}   Jumaniyozov D., Kaygorodov I.,   Khudoyberdiyev  A.,  
The algebraic  classification of nilpotent commutative algebras, 
Electronic Research Archive,    29  (2021),   6, 3909--3993.

 
 \bibitem{ikp20}
 Ignatyev M.,  Kaygorodov I., Popov Yu., 
  The geometric classification of $2$-step nilpotent algebras   and applications,   Revista Matemática Complutense, 
  35 (2022), 3, 907--922.
  
\bibitem{ikv19}
  Ismailov N., Kaygorodov I., Volkov Yu., 
  Degenerations of Leibniz and anticommutative algebras,  
  Canadian Mathematical Bulletin, 62 (2019), 3, 539--549.
 
\bibitem{japan}
  Kobayashi Yu., Shirayanagi K., Takahasi S.-Ei., Tsukada M.,  Classification of three-dimensional zeropotent algebras over an algebraically closed field,  
  Communications in Algebra, 45 (2017),  12, 5037--5052.

 
 
 \bibitem{ks21}
Kolesnikov P.,  Sartayev B.,
On the special identities of Gelfand-Dorfman algebras, 
Experimental Mathematics, 2022, DOI: 10.1080/10586458.2022.2041134.


 
   
 \bibitem{hom}
  Laraiedh I., Silvestrov S., 
  Transposed ${\rm Hom}$-Poisson and ${\rm Hom}$-pre-Lie Poisson algebras and bialgebras, arXiv:2106.03277
 

 

\bibitem{bihom}
 Ma T.,   Li B.,
Transposed ${\rm BiHom}$-Poisson algebras,
Communications in Algebra, 51 (2023),   2, 528--551.

 \bibitem{petersson}
Petersson H., 
The classification of two-dimensional nonassociative algebras, 
Results in Mathematics, 37 (2000), 1--2, 120--154.

 \bibitem{shaf}
    Shafarevich I., 
    Deformations of commutative algebras of class $2,$ Leningrad Mathematical Journal, 2 (1991), 6, 1335--1351.
    
      \bibitem{shirshov}
Shirshov A., 
Selected works of A. I. Shirshov,
Contemporary Mathematicians. Birkhäuser Verlag, Basel, 2009. viii+242 pp.
 
  \bibitem{sverchkov}
Sverchkov S., 
A quasivariety of special Jordan algebras, 
Algebra and Logic, 22 (1983),  5, 563--573.
 
 
 
 \bibitem{vd}
 Van den Bergh M., 
 Double Poisson algebras, 
 Transactions of the American Mathematical Society, 360 (2008), 11, 5711--5769.
 
 
 \bibitem{wolf2}
Volkov Yu., 
    $n$-ary algebras of the first level,
     Mediterranean Journal of Mathematics, 19 (2022), 1,  2.
 

  
 
 \bibitem{vv12}
Voronin V., 
Special and exceptional Jordan dialgebras, 
Journal of Algebra and its Applications,  11 (2012),  2, 1250029, 23 pp.


 

\bibitem{YYZ07}
Yao Y., Ye Y., Zhang P., 
Quiver Poisson algebras, 
Journal of  Algebra, 312 (2007), 2, 570--589.



    \bibitem{yh21}
Yuan L.,  Hua Q., 
$\frac{1}{2}$-(bi)derivations and transposed Poisson algebra structures on Lie algebras,
 Linear and Multilinear Algebra, 70 (2022),   22, 7672--7701.
 
 
 
 
 
\bibitem{z10}
 Zusmanovich P., 
 On $\delta$-derivations of Lie algebras and superalgebras, 
 Journal of Algebra, 324 (2010), 12, 3470--3486.
 
 

 
   
\end{thebibliography}
\end{document}